\documentclass[11pt]{article}
\usepackage{amsmath, amssymb, amsthm}
\usepackage{comment}
\usepackage{graphicx} 
\usepackage{enumitem}                                 
\usepackage[margin=1.0in]{geometry}

\newtheorem{theorem}{Theorem}[section]
\newtheorem{Remark} [theorem]{Remark}

\newtheorem{Counter-example}[theorem]{Counter example}
\newtheorem{Claim}[theorem]{Claim}

\newtheorem{Lemma}[theorem]{Lemma}
\newtheorem{Proposition}[theorem]{Proposition}

\newtheorem{Definition}[theorem]{Definition}

\newtheorem*{theorem*}{Theorem}

\newcommand{\supp}{\text{supp}}

\title{Spectral gaps and Fourier decay for self-conformal measures on the plane}
\author{Amir Algom, Federico Rodriguez Hertz, and Zhiren Wang}
\date{}
\begin{document}
\maketitle
\begin{abstract}
Let $\Phi$ be a $C^\omega (\mathbb{C})$ self-conformal IFS on the plane, satisfying some mild non-linearity and irreducibility conditions. We prove a uniform spectral gap estimate for the transfer operator corresponding to the derivative cocycle and every given self-conformal measure. Building on this result, we establish polynomial Fourier decay for any such measure. Our technique is based on a refinement of a method of Oh-Winter (2017) where we do not require separation from the IFS or the Federer property for the underlying measure.
 \end{abstract}

\section{Introduction}
\subsection{Fourier decay for self-conformal measures on the plane}
The Fourier transform of a Borel probability measure $\nu$  on $\mathbb{R}^2$ at $q$ is given by 
\begin{equation*} 
\mathcal{F}_q (\nu) := \int \exp( 2\pi i \langle q, x \rangle ) d\nu(x).
\end{equation*} 
This paper is concerned with the Fourier decay problem: is $\mathcal{F}_q (\nu)=o(1)$ as $|q|\rightarrow \infty$ (this is called the Rajchman property)? And, if this is the case, can one indicate a rate of decay? Of  particular interest is whether there is polynomial decay, that is, whether there exists some $\alpha>0$ such that  $
\left| \mathcal{F}_q \left( g \nu \right) \right| = O\left(\frac{1}{|q|^\alpha} \right)$. If such an $\alpha$ does exist, it is an even more challenging problem to try and evaluate it more precisely.

Beyond intrinsic interest, such estimates are known to have many applications in various fields. This includes e.g. the uniqueness problem in harmonic analysis \cite{li2019trigonometric}, the study of normal numbers \cite{Davenport1964Erdos}, spectral theory of transfer operators \cite{li2018fourier}, and the fractal uncertainty principle \cite{Bour2017dya}, to name just a few examples. We refer to Lyons's survey \cite{Lyons1995survey} for an overview on the Rajchman property, and to Mattila's book \cite{Mattila2015Fourier} for a general introduction to the role of Fourier analysis in geometric measure theory. The very recent survey of Sahlsten \cite{sahlsten2023fourier} contains a broad account of recent results and breakthroughs in the context of dynamically defined measures, which are the focus of our work. More specifically, in this paper we will study the Fourier decay problem for self-conformal measures on the plane; let us now recall their definition.

We first define self-conformal sets. Let $\Phi= \lbrace f_1,...,f_n \rbrace$ be an IFS, an abbreviation for Iterated Function System, defined on $\mathbb{C}$ (for us working in $\mathbb{C}$ or in $\mathbb{R}^2$ makes little difference). We assume each $f_i \in \Phi$ is a  strict injective contraction of the closure of an open set $D\subseteq \mathbb{C}$. In this paper we always assume every $f_i$ is $C^\omega (D)$, though IFSs that are only $C^\alpha$ smooth, $\alpha\geq 1$, are also of great interest. Under these assumptions, there is a unique compact set $\emptyset \neq K=K_\Phi \subseteq D$ so that
\begin{equation} \label{Eq union}
K = \bigcup_{i=1} ^n f_i (K).
\end{equation}
We call  $K$   the \textit{attractor} of $\Phi$ or the \textit{self-conformal set} corresponding to $\Phi$.  We always assume that $K$ is infinite; to this end, it suffices to suppose the fixed point of $f_i$ is not the same as the fixed point of $f_j$ for two elements in $\Phi$. This will thus be always assumed throughout the paper. We call the IFS $\Phi$   \textit{uniformly contracting} if 
$$0< \inf \lbrace |f '(x)|:\, f\in \Phi, x\in D \rbrace \leq \sup \lbrace |f '(x)| :\, f\in \Phi, x\in D \rbrace <1.$$
We proceed to define the coding map. Let  $\mathcal{A}= \lbrace 1,...,n\rbrace$, and let $\omega \in \mathcal{A} ^\mathbb{N}$.  For any $m\in \mathbb{N}$ let
$$f_{\omega|_m} := f_{\omega_1} \circ \dots \circ f_{\omega_m}.$$
Let us prefix some point $x_0 \in \overline{D}$.  The coding map $\pi$ is the surjective  map $\pi: \mathcal{A} ^\mathbb{N} \rightarrow K$ defined as
\begin{equation} \label{Eq coding}
\omega \in \mathcal{A}^{\mathbb N} \mapsto x_\omega:= \lim_{m\rightarrow \infty}  f_{\omega|_m}  (x_0).
\end{equation}

We now define self-conformal measures. Let $(p_1,...,p_n)$ be a probability vector; in particular, $\sum_{i=1} ^n p_i=1$, and we always assume each $p_i>0$. We will denote $(p_1,...,p_n)$ by $\textbf{p}$. Next, we define the infinite product measure $\mathbb{P}=\mathbf{p}^\mathbb{N}$  on $\mathcal{A}^\mathbb{N}$. The projected measure $\nu=\nu_\mathbf{p} = \pi \mathbb{P}$ onto $K$ is called the \textit{self conformal measure} corresponding to $\mathbf{p}$. The assumptions that $K$ is infinite and $\mathbf{p}$ is strictly positive ensure that $\nu$ is non-atomic.
 Alternatively, $\nu_\mathbf{p}$ can be defined as the unique Borel probability  measure supported on $K$ satisfying
$$\nu = \sum_{i=1} ^n p_i\cdot  f_i\nu,$$
where $f_i \nu$  is the push-forward of  $\nu$ by the map $f_i$. In the special case when every $f_i$ is affine,  $\Phi$ is called  a \textit{self-similar IFS}. In this case we call $\nu$ a \textit{self-similar measure}.

We are now in position to state the main result of this paper. We say that a $C^r$ IFS $\Phi$ is  conjugate to an IFS $\Psi$ if they are the same up to a $C^r$ change of coordinates; that is, if $\Phi = \lbrace h\circ g \circ h^{-1}\rbrace_{g\in \Psi}$ where $h$ is a $C^r$ diffeomorphism. 
\begin{theorem} \label{Main Theorem analytic}
Let $\Phi$ be a uniformly contracting $C^{\omega} (D)$ IFS, defined on the closure of the unit disc $D=\overline{B_1(0)}$. Assume:
\begin{enumerate}
\item  $\Phi$ is not conjugate to a self-similar IFS; and

\item  $K_\Phi$ is not a subset of an analytic planar curve.
\end{enumerate}
Then for every self-conformal measure $\nu$ there exist  $\alpha=\alpha(\nu)$ and $ C=C(\nu)>0$ satisfying
\begin{equation*}
\left| \mathcal{F}_q \left(  \nu \right) \right| \leq C\cdot \frac{1}{|q|^\alpha}.
\end{equation*} 
\end{theorem}
Before putting Theorem \ref{Main Theorem analytic} in the context of recent research, let us first discuss its various assumptions.
\begin{enumerate}
\item  Regarding the domain,  recall that such IFSs may be defined more generally on the closure of an open set, whereas we restrict ourselves to IFSs defined on  closed discs. This is  done for technical convenience as the proof is already quite involved even in this case. It is likely that by combining our technique with that of Leclerc \cite{Leclerc2023julia} (specifically, his use of the Koebe $1/4$-Theorem) one can significantly relax this assumption, at least to the closure of a simply connected domain. This seems to be the minimal requirement since at several points in the proof we require the existence of branches of  $\log f'$ for certain diffeomorphisms $f$, e.g. in Claim \ref{Claim UNI}.

\item The assumption that $\Phi$ is not conjugate to  self-similar is formally necessary since the conjugating map may be affine, in which case the result might be false, see \cite{rapaport2021rajchman}. Nonetheless, in dimension $1$ it is known that a strictly convex $C^2$ image of a self-similar measure always has polynomial decay (this was proved in increasing generality by Kaufman \cite{Kaufman1984ber}, Mosquera-Shmerkin \cite{Shmerkin2018mos}, and in the  general case independently by Algom-Chang-Wu-Wu \cite{Algom2023Wu} and Baker-Banaji \cite{baker2024polynomial}). Thus, in dimension $1$ it is known that just having one non-affine map in a $C^\omega$ IFS suffices for polynomial decay \cite{algom2023polynomial, baker2024polynomial, Algom2023Wu}. In the plane, the only currently known version of Kaufman's Theorem \cite{Kaufman1984ber} is for homogeneous (equicontractive) self-similar IFSs, due to Mosquera-Olivo \cite{Mos2023fourier}. If this result could be extended to all irreducible self-similar measures (i.e. those that don't live on hyperplanes) then we can replace assumption (1) in Theorem \ref{Main Theorem analytic} with the assumption that $\Phi$ contains a non-affine map, similarly to the one dimensional case. In fact, after we completed this paper we became aware of an on-going project by Banaji and Yu \cite{Banaji2024Han} that aims to study such higher dimensional versions of Kaufman's Theorem.

\item The assumption that $K_\Phi$ is not a subset of an analytic planar curve is also formally necessary, since the conclusion of Theorem \ref{Main Theorem analytic} clearly fails if $K_\Phi$ belongs to a hyperplane. However, in the presence of some curvature the situation becomes much more interesting. Here, following the initial work of Orponen \cite{Orponen2023add}, there have been several papers \cite{Dasu2024demeter, demeter2024szemer, orponen2024fourier}  that prove quantitative Fourier decay estimates \textit{on average} for (not necessarily dynamically defined) measures supported on such curves; however, we are not aware of  such  examples  where  \textit{pointwise} estimates have been established.

\end{enumerate}

Let us now place Theorem \ref{Main Theorem analytic} within current literature. It is a planar counterpart of recent results obtained simultaneously by the authors and Baker-Sahlsten \cite{algom2023polynomial, Baker2023Sahl} for IFSs on the line. Prior work on the Fourier decay problem had always assumed some additional algebraic, dynamical, or fractal geometric assumptions, and \cite{algom2023polynomial, Baker2023Sahl} were essentially the first results to achieve polynomial decay (and spectral gap that we discuss in the next section) in a stripped down axiomatic setting. Previous related results include the work of Jordan-Sahlsten \cite{Sahl2016Jor} about invariant measures for the Gauss map;  Bourgain-Dyatlov \cite{Bour2017dya} and Khalil \cite{khalil2023exponential} for some Patterson-Sullivan measures;  the various works by Li \cite{li2018fourier, Li2018decay} about Furstenberg measures; the work of Sahlsten-Stevens \cite[Theorem 1.1]{sahlsten2020fourier} about a class of stationary measures with respect to non-linear $C^\omega (\mathbb{R})$ IFSs with strong separation (that is,  the union \eqref{Eq union} is disjoint); and,  it improved our previous result \cite[Theorem 1.1]{algom2021decay} by improving the decay rate  from logarithmic to polynomial. We do note that \cite{algom2023polynomial, Baker2023Sahl} say nothing about self-similar measures, which is an important case. For recent progress on decay for self-similar measures we refer to \cite{Dai2007Feng, Solomyak2021ssdecay, varju2020fourier,   algom2020decay, Shmerkin2018mos, Buf2014Sol,  Dai2012ber, streck2023absolute, rapaport2021rajchman} and references therein.

Theorem \ref{Main Theorem analytic} also has counterparts in higher dimensions. With the exception of the very recent work of Baker-Khalil-Sahlsten \cite{khalil2024polynomial}, these results all prove polynomial decay in various more specialized algebraic or geometric settings. We remark that a version Theorem \ref{Main Theorem analytic} for Gibbs measures projected to the attractor assuming only a positive Lyapunov exponent  would give a unified proof of many of these results; this is left for future research. Now, Li-Naud-Pan \cite{Li2021Naud} extended the results of Bourgain-Dyatlov \cite{Bour2017dya} for limits of general Kleinian Schottky groups. This was later further generalized by Baker-Khalil-Sahlsten \cite{khalil2024polynomial}, relaying only on additive methods rather than on sum-product estimates. Leclerc, building on the methods of Sahlsten-Stevens and Oh-Winter \cite{Oh2017winter}, proved polynomial decay for Juila sets for hyperbolic rational maps of $\hat{ \mathbb{C}}$. We note that Theorem \ref{Main Theorem analytic} has non-trivial yet non-full overlap with virtually all of these results; and, if the Gibbs-version previously alluded to can be established then it would formally cover all of them. Very recently, Baker-Khalil-Sahlsten \cite[Theorems 1.14 and 1.17]{khalil2024polynomial} obtained polynomial decay for certain stationary measures on self-conformal IFSs, also in higher dimensions. Some of their assumptions are less restrictive than ours, e.g. they only ask that the IFS be $C^2$ and can treat the more general class of Gibbs measures; but some are more restrictive, e.g. they require the underlying IFS to be strongly separated and that the measures be affinely non-concentrated (see e.g. \cite{Urbanski2005dio}). The paper \cite{khalil2024polynomial} also contains results regarding Fourier decay for carpet like non-conformal IFSs; these seem to have no formal relations with Theorem \ref{Main Theorem analytic}.

Finally, we briefly survey some known results about Fourier decay for self-similar and self-affine measures in higher dimensions. This includes the work of Li-Sahlsten \cite{Li2020Sahl} and Solomyak \cite{solomyak2021fourier} for self-affine measures, the work of Rapaport \cite{rapaport2021rajchman} classifying those self-similar measures that are Rajchman, and the work of Varj\'{u} about poly-logarithmic decay of self-similar measures discussed in \cite{sahlsten2023fourier}. Baker-Khalil-Sahlsten \cite{khalil2024polynomial} recently obtained new examples of self-similar measures with poly-logarithmic decay assuming certain Diophantine conditions.

\subsection{Spectral gap via a cocycle version of a technique of Oh-Winter, and the $C^2$ case} \label{Section method}
The key step towards  Theorem \ref{Main Theorem analytic} is the proof of a spectral gap type estimate for the derivative cocycle. This result, which is the main technical result of our paper, is of  intrinsic interest due to the many applications it has. In our case, we will use it to show that certain random walks that arise by re-normalizing the derivative cocycle equidistribute exponentially fast towards an absolutely continuous measure (Theorem \ref{Theorem equi}); this  is the key step towards Theorem \ref{Main Theorem analytic}. More precisely, our spectral gap estimate (Theorem \ref{Theorem spectral gap}) implies a renewal theorem with exponential speed (Proposition \ref{Proposition renewal}) in the spirit of Li \cite{li2018fourier}, and it is this result that leads to the desired equidistribution result Theorem \ref{Theorem equi}. The proof of Fourier decay given Theorem \ref{Theorem equi} is similar to \cite{algom2020decay, algom2023polynomial}; and the proof of Theorem \ref{Theorem equi} given the renewal Theorem with exponential speed is similar to e.g. the work of Li-Sahlsten \cite{Li2020Sahl} or to our work \cite{algom2023polynomial}. So, we focus our attention in this sketch on the proof of the spectral gap, which is our chief innovation.

Our  strategy is adapted, in general terms, from our one dimensional paper \cite{algom2023polynomial}. However, there are deep changes due to the higher dimensional setup. Thus, it would be instructive to first briefly recall our method from \cite{algom2023polynomial}.

So, assuming say condition (1) of Theorem \ref{Main Theorem analytic} holds for a one dimensional real-analytic IFS, we aimed to prove spectral gap (Theorem \ref{Theorem spectral gap}) for the transfer operator defined by the norm derivative cocycle and $\mathbf{p}$ (\eqref{The der cocycle} and Definition \ref{Def transfer operator}). This is based on Dolgopyat's method \cite{Dol1998annals, Dolgopyat2000Mix2} as originally brought into the setting of self-conformal sets by Naud \cite{Naud2005exp}, see also \cite{Stoyanov2011spectra, Stoyanov2001decay, Araujo2016Melbourne, Avila2006yoccoz, Baladi2005Valle}. However, given our stripped down setting, we were met with three main obstacles: the union \eqref{Eq union} may not correspond to a Markov partition; the doubling (Federer) property might not be satisfied by  the measure $\nu$; and that $\Phi$ must the satisfy Dolgopyat's (UNI) condition, see Claim \ref{Claim UNI}). We note that condition (1) suffices to ensure that (UNI) does indeed hold, which follows from our discussion in \cite{algom2021decay}. The other obstructions are, however, much more substantial.

To get around them, we introduced  a cocycle version of Dolgopyat's method: we wrote $\nu$ as an integral over a certain family of random measures, and re-wrote the transfer operator as a corresponding weighted sum over smaller parts of itself. This was based upon a technique from \cite{Algom2022Baker}, though it had been used in several previous papers with various goal in mind \cite{Galcier2016Lq, Antti2018orponen, Shmerkin2018Solomyak, solomyak2023absolute}. Here, these  measures exhibit some stationary behaviour,  have strong separation in an appropriate sense, and have the doubling  property. The (UNI) condition was satisfied by all the different pieces of the transfer operator in this decomposition, which was critically important. We were thus able to run Naud's version of Dolgopyat's method by "jumping" between the various conditional measures (hence the term "cocycle").

In the present paper a major departure from \cite{algom2023polynomial} is that now the derivative cocycle comprises of the norm cocycle \eqref{The der cocycle} (as before) but also of an angle cocycle \eqref{The angle cocycle}. In particular, the transfer operator, in Definition \ref{Def transfer operator}, also takes the angle into account, and "couples" the two cocycles.  To prove spectral gap we rely on the  analysis of Oh-Winter  \cite{Oh2017winter} that was done for certain Julia sets. However, we are again faced with the lack of a Markov partition and the lack of the Federer property for $\nu$. Furthermore, we also require what Oh-Winter refer to as the NCP  (non-concentration property), see e.g. Theorem \ref{Theorem disint} Part (6), which is a type of UNI but for the angle cocycle. It is not hard to show that the set $K_\Phi$ satisfies this condition due to assumption (2) in Theorem \ref{Main Theorem analytic}. 

To get around these issues we construct conditional measures and disintegrate the transfer operator accordingly. As the NCP  for the angle cocycle plays a crucial role for the spectral gap argument, we must preserve it into every corresponding part of the transfer operator, while not losing separation or (UNI), which is non-trivial. The construction of these random measures and pieces  takes place in the first few subsections of Section \ref{Section spectral gap}. After a standard $L^2$ reduction, the spectral gap result is reduced to the construction of certain Dolgopyat operators (Lemma \ref{Lemma 5.4 Naud}). Its proof,  in Section \ref{Section proof}, which is the heart of the  technical work in this paper, is a cocycle version of a method of Oh-Winter \cite[Sections 5.3 - 5.5]{Oh2017winter}.

Finally, let us discuss our $C^\omega$ assumption compared with the $C^2$ assumptions we make in \cite{algom2023polynomial}. The issue is that to run the Oh-Winter proof scheme, that is, to obtain spectral gap, one requires something a-priori stronger than UNI. Namely, that a certain coupling of the norm and the angle cocycles defines a local diffeomorphism from small balls in $D$ to $\mathbb{R}\times \mathbb{T}$; and, that moreover the $C^2$ norm of this map can be made uniform across the disc. This is Lemma \ref{Lemma 3.2 Oh}, which critically uses the Cauchy-Riemann equations to derive information about the angle from information about the norm. While we believe the general outcome of Theorem \ref{Main Theorem analytic} should follow if one assumes the outcome of this Lemma in the $C^2$ case, we do not see how to deduce it simply from a non-linearity condition as in the $1$-dimensional case.

\section{Spectral gap} \label{Section spectral gap}
\subsection{Some notations and conventions} 
Throughout this paper we denote  by $C^{\omega} (D)$ the family of holomorphic functions on  $D$, where $D$ is  the closure of the unit disc. Often, we will work with real valued functions, in which case the gradient of $h:\mathbb{C}=\mathbb{R}^2 \rightarrow \mathbb{R}$ will be denoted by $\nabla h :\mathbb{C}=\mathbb{R}^2 \rightarrow \mathbb{R}$. The Euclidean norm on $\mathbb{R}^2$ will be denoted by $|\cdot |$. We remark that occasionally we will consider complex valued functions that may not be holomorphic. For these functions, we use the notation $\nabla h$ to mean the Jacobian of $h$ as a map $\mathbb{R}^2 \rightarrow \mathbb{R}^2$, and $|\nabla h|$ denotes the operator norm of the matrix $\nabla h$. 

Finally, in several places in the proof we will require the following version of the mean value Theorem: For every $g\in C^\omega (D)$ and $z,w\in D$ there exists some $v\in D$ such that
\begin{equation} \label{MVT}
\left| g(z)-g(w) \right| \leq |g'(v)|\cdot |z-w|.
\end{equation}

\subsection{Preliminaries} \label{Section induced IFS}
Let $\Phi = \lbrace f_1,...,f_n \rbrace$ be a $C^{\omega} (D)$ IFS, and recall that $\mathcal{A}= \lbrace 1,...,n\rbrace$. In particular,
$$f_i(D) \subseteq D,\quad \forall f_i \in \Phi.$$
We write
$$\rho_{\min} := \min_{i\in \mathcal{A}, x\in D} |f_i ' (x)|.$$
There exists some $L=L(\Phi)$ such that 
\begin{equation} \label{Eq bdd distortion}
L^{-1} \leq \left|f_\eta ' (x)\right|/\left|f_\eta ' (y)\right|\leq L \text{ for all } x,y\in D \text{ and every } \eta \in \mathcal{A}^*.
\end{equation}
This is called the bounded distortion property, see e.g. \cite[Lemma 2.1]{algom2020decay}.

Next, arguing similarly to  \cite[equation (5)]{algom2023polynomial} and using  the Cauchy-Riemann equations,   we may assume that for some $\tilde{C}>0$
\begin{equation} \label{Eq tilde C}
\sup_{x\in D, \xi\in \mathcal{A}^\mathbb{N}, n\in \mathbb{N}}  \left| \nabla \left( \log \left| f' _{\xi|_n} \right| \right)(x) \right| \leq \tilde{C}.
\end{equation}

For $N \in \mathbb{N}$ we define the IFS $\Phi^N$ by
$$\Phi^N := \lbrace f_{\eta}:\, \eta\in \mathcal{A}^N\rbrace.$$
We call $\Phi^N$ the $N$-generation induced IFS. Note that the same $L>0$ from \eqref{Eq bdd distortion} works equally as well for any induced IFS of any generation.

The following Claim will be instrumental when we construct the disintegration in Section \ref{Section dis}.
\begin{Claim} \label{Claim key}
Let $\Phi$ be a $C^\omega (D)$ IFS satisfying the assumptions of Theorem \ref{Main Theorem analytic}. Then there is an induced IFS $\Phi^N, N\in \mathbb{N}$ and four maps $f_1,f_2,f_3,f_4 \in \Phi^N$  such that:
\begin{enumerate}
\item Let $k\geq 5$. Then there is some $i\in \lbrace 1,3\rbrace$ such that 
$$f_i (D) \cup  f_{i+1} (D) \cup f_k (D) \text{ is a disjoint union.}$$
In addition, $f_i (D) \cup  f_{i+1} (D)$ is also a disjoint union for $i=1,3$.

\item There are  $m',m>0$ satisfying: For both $i=1,3$ and every $x\in D$, 
$$m\leq \left|\nabla  \left(  \log \left| f_{i } ' \right| - \log \left| f_{i+1} ' \right| \right) \left(x\right) \right|\leq m'.$$

\item The following estimate holds:
$$m-2\cdot \tilde{C}\cdot  \left( \sup_{i} ||f_i '||_\infty \right) ^N >0.$$
\end{enumerate}
\end{Claim}
We note that the assertion about the upper bound in Part (2) is well known and follows from e.g. \eqref{Eq tilde C}.

Claim \ref{Claim key} is an analogue of \cite[Claim 2.1]{algom2023polynomial}. Their proofs are similar, except for the following Claim which requires special attention due to our higher dimensional setting. 
\begin{Claim} \label{Claim UNI}
We can find $c,m',N_0>0$ so that for every $n>N_0$ there are $\xi,\zeta \in \mathcal{A}^n$  such that for  all $x\in D$,
$$c< \left|\nabla  \left(  \log \left| f_{\xi|_n }  ' \right| - \log \left| f_{\zeta|_n }   ' \right| \right) \left(x\right) \right|\leq m'.$$
\end{Claim}
The proof is both based on and uses \cite[Claim 2.13]{algom2021decay}, which is its one dimensional counterpart. However, there are some non-trivial new subtleties, so we proceed with care. We remind the reader of our assumptions: $\Phi$ is not conjugate to a self-similar IFS, and $K_\Phi$ is not a subset of an analytic planar curve.
\begin{proof}
In this proof we work with  branches of the complex $\log f'$ defined on $D$ for certain diffeomorphisms $f$. They exists here since $D = \overline{B_1(0)}$ by our assumption.

Fix some $x_0 \in D$. For every $\xi \in \mathcal{A}^\mathbb{N}$ we define a function $f_\xi:D\rightarrow \mathbb{C}$ via 
$$\varphi_\xi = \lim_{n\rightarrow \infty} \log f_{\bar{\xi}|_n} '(x) - \log f_{\bar{\xi}|_n} '(x_0),\, \text{ where } \bar{\xi}:=(....\xi_{n-1},...,\xi_2,\xi_1)\in \mathcal{A}^{-\mathbb{N}}.$$ 
Just to be clear, $f_{\bar{\xi}|_n} = f_{\xi_{n}}\circ \dots \circ f_{\xi_1}$. Let us first discuss the properties of this function. Similarly to the argument in \cite[Claim 2.13]{algom2021decay}, it is readily shown that for every $\xi \in \mathcal{A}^\mathbb{N}$ and $a\in \mathcal{A}$,
\begin{equation} \label{Eq iterate}
\varphi_{a*\xi} = \log f_a ' + \varphi_\xi \circ f_a+c_{a*\xi}.
\end{equation}
Here, "$*$" denotes concatenation of words (or of a word with a sequence), and $c_{a*\xi} \in \mathbb{C}$ is some constant. We also note that $\varphi_\xi \in C^\omega (D)$ and that the association $\xi \mapsto \varphi_\xi$ is continuous. Finally, we remark for future use that the definition of $\varphi_\xi$ coincides with the definition given in \cite[Claim 2.13]{algom2021decay} when $\xi$ is periodic (up to the removing the bar from $\bar{\xi}$ in its definition).

We next claim that there exist some $\xi,\zeta \in \mathcal{A}^\mathbb{N}$ and some $x_1\in K$ such that
$$\Re( \left( \varphi_\xi- \varphi_\zeta \right)' )\cdot \Im( \left( \varphi_\xi- \varphi_\zeta \right)')(x_1)\neq 0.$$
Indeed, suppose towards a contradiction that this is not the case. Notice that $$x\mapsto \Re( \left( \varphi_\xi- \varphi_\zeta \right)' )\cdot \Im( \left( \varphi_\xi- \varphi_\zeta \right)')(x)$$
is an harmonic function, as it is equal to
$$\frac{1}{2} \Im( \left( \varphi_\xi- \varphi_\zeta \right)' )^2 )(x).$$
So, our assumption towards a contradiction implies that $K$ lies on a level set (corresponding to the value $0$) of this harmonic function. Since we are assuming that this is not the case, we conclude that for every $\xi,\zeta \in \mathcal{A}^\mathbb{N}$
$$ \Im( \left( \varphi_\xi- \varphi_\zeta \right)' )^2 )(x) \equiv 0 \text{ on } D.$$
This in turn implies that the holomorphic function $\left( \varphi_\xi- \varphi_\zeta \right)' )^2$ is a constant real number for every $\xi,\zeta \in \mathcal{A}^\mathbb{N}$. Thus, $\varphi_\xi- \varphi_\zeta$ is an affine map for every such pair of parameters.

Our next step is to show that this affine map is actually identically zero. Indeed, let $\xi,\zeta \in \mathcal{A}^\mathbb{N}$ and let $a\in \mathcal{A}$.  Applying \eqref{Eq iterate} twice, we see that
$$\varphi_{a*\xi} = \log f_a ' + \varphi_\xi \circ f_a+c_{a*\xi}, \text{ and}$$
$$\varphi_{a*\zeta} = \log f_a ' + \varphi_\zeta \circ f_a+c_{a*\zeta}.$$
So,
$$\varphi_{a*\xi} - \varphi_{a*\zeta} = \left( \varphi_\xi- \varphi_\zeta \right)\circ f_a+d_{\xi,\zeta,a},\,\text{ where }  d_{\xi,\zeta,a}\in \mathbb{C}.$$
Now, by our assumption there exists some $f_a \in \Phi$ that is not affine. Therefore, as  $\varphi_{a*\xi} - \varphi_{a*\zeta}$ is affine, $ \varphi_\xi- \varphi_\zeta$ must be constant. Since both maps have a common zero at $x_0$ we conclude that this constant is $0$. 

Thus, our assumption towards a contradiction has led us to conclude that for every $\xi,\zeta \in \mathcal{A}^\mathbb{N}$ we have
$$\varphi_\xi = \varphi_\zeta.$$
This is in particular true for $1$-periodic elements. Thus, we may now argue as in \cite[Claim 2.13]{algom2021decay} that  the IFS $\Phi$ is $C^\omega$ conjugate to a $C^\omega$ linear $\Psi$ IFS, i.e., one such that $g''=0$ on its attractor $K_\Psi$ for $g \in \Psi$. Since $\Psi$ is $C^\omega$ this forces $\Psi$ to be self-similar. We thus see that $\Phi$ is conjugate to self-similar, contradicting our assumptions.

So, we have shown that there exist some $\xi,\zeta \in \mathcal{A}^\mathbb{N}$ and some $x_1\in K$ such that
$$\Re( \left( \varphi_\xi- \varphi_\zeta \right)' )\cdot \Im( \left( \varphi_\xi- \varphi_\zeta \right)')(x_1)\neq 0.$$


Thus,  there are $c'>0$, $x_1\in K$ such that for all $n$ there are words  $\alpha=\bar{\xi}|_n, \beta= \bar{\zeta}|_n \in \mathcal{A}^n$  such that 
$$c'< \left| \nabla  \left(  \log \left| f_{\alpha }  ' \right| - \log \left|   f_{\beta }   ' \right| \right) \left(x_1\right) \right|.$$
Note that here we are taking the usual forward composition when considering $\alpha,\beta$. Let $\omega \in \mathcal{A}^\mathbb{N}$ be a coding of $x_1$, that is, $x_\omega = x_1$ (recall \eqref{Eq coding}). Note that we have uniform convergence  as $n\rightarrow \infty$   of the series
$$ \left|\nabla \left(  \log \left| f_{\bar{\xi}|_n }  ' \right| - \log \left|   f_{\bar{\zeta}|_n }   ' \right| \right) \left(\cdot \right) \right|.$$
Thus, for some $N_1$ and  $k=k(N_1,c')$ we have for every $n>N_1$ and for the corresponding words $\alpha,\beta$ of length $n$
$$\left|\nabla  \left(  \log \left| f_{\alpha }  ' \right| - \log \left|  f_{\beta }  ' \right| \right) \left( x_1 \right) \right| -\frac{c'}{2} \leq \left|\nabla  \left(  \log \left|  f_{\alpha }  ' \right| - \log \left| f_{\beta }  ' \right| \right) \left( f_{\omega|_k}(x)\right) \right| $$
Let $n>\max\lbrace  N_1,N_0 \rbrace$ and $k=k(N_1,c')$ be chosen as indicated above. Then for every $x\in D$,  for the corresponding words $\alpha,\beta$ of length $n$
$$\frac{c'}{2} \cdot  \rho_{\min} ^k \leq  \left|\nabla  \left(  \log \left|  f_{\alpha }  ' \right| - \log \left| f_{\beta }  ' \right| \right) \left( f_{\omega|_k}(x)\right) \right| \cdot \left| f_{\omega|_k}'(x) \right|$$
$$=  \left|\nabla  \left(  \log \left| \left( f_{\alpha } \circ f_{\omega|_k} \right) ' \right| - \log \left| \left(  f_{\beta } \circ f_{\omega|_k}  \right) ' \right| \right) \left(x\right) \right|.$$
Therefore,  $c=\frac{c'}{2} \cdot  \rho_{\min} ^k$,  $\xi' = \alpha*\omega|_{k}$, and $\zeta' = \beta *\omega|_{k}$, satisfy the assertion of the Lemma for all $n>\max\lbrace  N_1,N_0 \rbrace$ (recall that we don't have to worry about the upper bound). The proof is complete.
\end{proof}
$$ $$
From this point, the proof of Claim \ref{Claim key}  follows similarly to \cite[proof of Claim 2.1]{algom2023polynomial}. We omit the details.

 Let us now fix a strictly positive probability vector $\textbf{p}=(p_1,...,p_n)$  on $\mathcal{A}$, and as usual we consider the product measure $\mathbb{P}=\mathbf{p}^\mathbb{N}$  on $\mathcal{A}^\mathbb{N}$. We consider the self-conformal measure $\nu=\nu_\mathbf{p}$.

It is standard that for all $N\in \mathbb{N}$ we have $K_{\Phi^N} = K_\Phi$. That is, the original IFS shares the same attractor as all induced IFSs of any generation. Moreover, the measure $\nu$ is a self-conformal measure w.r.t $\Phi^N$ and the corresponding probability vector $\mathbf{p}^N$ on $\mathcal{A}^N$. Therefore, by Claim \ref{Claim key}, moving to an induced IFS $\Phi^N$ for a suitable $N$, we may assume that already in our IFS $\Phi$ there maps  $f_1,f_2,f_3,f_4 \in \Phi$ such that:

\begin{enumerate}
\item Let $k\geq 5$. Then there is some $i\in \lbrace 1,3\rbrace$ such that \begin{equation} \label{Add property 1}
f_i (D) \cup  f_{i+1} (D) \cup f_k (D) \text{ is a disjoint union.}
\end{equation}
In addition, $f_i (D) \cup  f_{i+1} (D)$ is also a disjoint union for $i=1,3$.

\item There are  $m',m>0$ satisfying: For both $i=1,3$ and every $x\in D$, 
\begin{equation} \label{Add property 2}
m\leq \left|\nabla  \left(  \log \left| f_{i } ' \right| - \log \left| f_{i+1} ' \right| \right) \left(x\right) \right|\leq m',\text{ and } m-2\cdot \tilde{C}\cdot \max_{i\in A} ||f'_i||_\infty>0.
\end{equation}
\end{enumerate}

We proceed to define  the various derivative cocycles and corresponding transfer operator. The (free) semigroup $G$ that is  generated by  $\{f_a: 1\leq a\leq n\}$ acts on $D$ by composition in the obvious way. We define the norm (derivative) cocycle $c:G\times D\rightarrow \mathbb{R}$ via
\begin{equation} \label{The der cocycle}
c(I,x)=-\log \left| f'_I (x) \right|,
\end{equation}
and the angle (derivative) cocycle $\theta:G\times D\rightarrow \mathbb{R}$ via  
\begin{equation} \label{The angle cocycle}
\theta(I,x)=\arg \left( f'_I (x) \right) \in \mathbb{T}.
\end{equation}
Thus for all $x\in D$ and $i\in \mathcal{A}$ we have
$$f_i '(x) = e^{ -c(i,x)+i\theta(i,x)}.$$

We can now define the transfer operator. Note that it differs from the one we use in \cite[Definition 2.3]{algom2023polynomial}, since the angle cocycle plays a role here as well. It is modelled after the operator used by Oh-Winter \cite[Section 2.3]{Oh2017winter}.
 \begin{Definition} \label{Def transfer operator}
 Let $s \in \mathbb{C}$ be such that $|\Re(s)|$ is sufficiently small. The transfer operator $P_{s}:C^1(D) \rightarrow C^1(D)$ is defined by:
 
For $g \in C^1(D)$, $x \in D$, and a character $\chi:S^1 \rightarrow S^1$,
$$P_{s,\chi} (g)(x) = \int e^{2\pi \cdot s\cdot  c(a,x)} \chi\left( e^{i\theta(a,x)} \right) g\circ f_a(x)\, d\mathbf{p}(a). $$
If the character has the form $\chi_\ell (z)=z^\ell, \ell\in \mathbb{Z},$ we sometimes denote
$$P_{s,\ell}:=P_{s,\chi_\ell}.$$
In particular, 
$$P_s:=P_{s,0}.$$
 \end{Definition}
Since $\nu$ is the (unique) stationary measure for the measure $\sum_{a\in \mathcal{A}} p_a \cdot \delta_{\lbrace f_a \rbrace}$ on $G$, the general discussion
about transfer operators as in \cite[Section 11.5]{Benoist2016Quint}  can be applied here.

Finally, we define 
$$\rho:= \sup_{f\in \Phi} ||f'||_\infty \in (0,1).$$
And,  by uniform contraction we have
\begin{equation} \label{Eq C and C prime}
0<D:= \min \lbrace -\log |f' (x)| : f\in \Phi, x\in D \rbrace, \quad D':= \max \lbrace -\log |f' (x)| : f\in \Phi, x\in D \rbrace < \infty.
\end{equation}

\subsection{A model (disintegration) construction} \label{Section dis} 
In this Section we construct the disintegration of our self-conformal measure $\nu$ for the IFS $\Phi$, and the transfer operator $P_{s,\ell}$. We follow the notations as in Section \ref{Section induced IFS}. This is based on the notion of a model, similarly to \cite{Shmerkin2016Solomyak}. Our treatment is a two dimensional version of that in \cite[Section 2.3]{algom2023polynomial}

Fix a finite set $I$  of $C^\omega (D)$ IFSs $\lbrace f_1^{(i)},...,f_{k_i} ^{(i)} \rbrace, i\in I$.  We define our parameter space as $\Omega = I^\mathbb{N}$. Given $\omega \in \Omega$ and $n\in \mathbb{N}\cup \lbrace \infty \rbrace$ we define
$$X_n ^{(\omega)} := \prod_{j=1} ^n \lbrace 1,...,k_{\omega_j} \rbrace.$$
We now define a relative coding map $\Pi_\omega : X_\infty ^{(\omega)} \rightarrow \mathbb{C}$ by
$$\Pi_\omega (u) = \lim_{n} f^{ (\omega_1)} _{u_1} \circ \dots \circ f^{ (\omega_n)} _{u_n} (x_0)\, \text{ for some } x_0\in D.$$
We proceed to put measures on these objects. Suppose that for every  $i\in I$ we are given a strictly positive probability vector $\mathbf{p}_i = (p_1 ^{(i)},...,p_{k_i} ^{(i)})$. Then for every $\omega \in \Omega$ we can define a product measure on  $X_\infty ^{(\omega)}$ via
$$\eta^{(\omega)} := \prod_{n=1} ^\infty \mathbf{p}^{(\omega_n)}.$$
Projecting these measures by the coding map, we define
$$\mu_\omega:= \Pi_\omega \left( \eta^{(\omega)} \right). $$
Consider the left shift $\sigma:\Omega \rightarrow \Omega$. Then we have a stochastic stationarity relation:
\begin{equation} \label{Eq stochastic self-similarity}
\mu_\omega = \sum_{j\in X_1 ^{(\omega)}} p_j ^{(\omega_1)}\cdot  f_j ^{(\omega_1)} \mu_{\sigma(\omega)}.
\end{equation}
For every $\omega \in \Omega$ we have corresponding attractors
$$K_\omega := \supp(\mu_\omega) = \Pi_\omega \left( X_\infty ^{(\omega)} \right).$$
Similarly to \eqref{Eq stochastic self-similarity} we have
$$K_\omega = \bigcup_{j\in  X_1 ^{(\omega)}}  f_j ^{(\omega_1)} K_{\sigma(\omega)}.$$

Let us now fix $\omega\in \Omega, N, \ell \in \mathbb{N}$,  and $s\in \mathbb{C}$. The operator $P_{s, \ell, \omega,N}:C^1\left(D\right)\rightarrow C^1\left(D\right)$ is defined by
\begin{equation} \label{Eq dis operator}
P_{s, \ell, \omega,N} \left( g \right) (x):= \sum_{J\in X_N ^{(\omega)}} \eta ^{(\omega)} (J) e^{2\pi s c(J,x)} \chi_\ell \left( e^{i \theta\left(J,x\right)}\right)  g\circ f_J (x), \, \text{ where we put } \eta ^{(\omega)} (J) := \prod_{n=1} ^{N} \textbf{p}^{(\omega_n)}_{  J_n }.
\end{equation}

Let us now put a selection measure $\mathbb{Q}$, which is by definition $\sigma$-invariant,  on $\Omega$. We call the triplet $\Sigma = (\Phi_{i\in I} ^{(i)}, (\mathbf{p}_i)_{i\in I}, \mathbb{Q})$ a model. The model is called non-trivial if the measures $\mu_\omega$ are non-atomic almost surely, and if $\mathbb{Q}$ is a Bernoulli measure  it is called Bernoulli. The following Model construction will play a key role in our analysis:
\begin{theorem} \label{Theorem disint}
Let $\Phi$ be a  $C^\omega (\mathbb{C})$ IFS satisfying \eqref{Add property 1} and \eqref{Add property 2}. In addition, assume $K_\Phi$ does not belong to an analytic curve. Then for any self-conformal measure $\nu=\nu_{\mathbf{p}}$ there is a non-trivial Bernoulli model $\Sigma=\Sigma(\Phi, \nu_{\mathbf{p}})$ so that:
\begin{enumerate}
\item We can disintegrate the measure $\nu$ as 
$$\nu = \int \mu_{\omega} d\mathbb{Q}(\omega).$$

\item  For all $N, \ell \in \mathbb{N},s\in \mathbb{C}, f\in C^1(D),$ and every $x\in D$ we can disintegrate $P_{s,\ell} ^{N}$ via
\begin{equation} \label{Dis of transfer}
P_{s,\ell} ^{N} \left(f \right)(x) = \sum_{\omega \in \Omega^{N}} \mathbb{Q}\left([\omega]\right) P_{s, \ell, \omega, N} \left( f \right)(x).
\end{equation}

\item  For all $j\in I$ we have $k_j \geq 2$ (the model branches non-trivially).

\item For all  $\omega \in \Omega$ 
$$\bigcup_{ j\in X_1 ^{(\omega)}} f_j ^{(\omega_1)} (D)$$
is a disjoint union (the model has separation).

\item   For some $m',m>0$ and $N_0\geq 0$ and every $N\geq N_0$, for all $\omega \in \Omega$ we can find $\alpha_1 ^N, \alpha_2 ^N \in X_N ^{(\omega)}$ satisfying
 
$m \leq \left|\nabla  \left(  \log \left| f_{\alpha_1 ^N } ' \right| - \log \left| f_{\alpha_2 ^N } ' \right| \right) \left(x\right) \right| \leq m', \quad \text{ for every } x\in D$. (UNI is preserved in all branches).

\item (Non concentration in all parts) There exists $\delta>0$ such that: for every $\omega \in \Omega$ and any  direction $w\in S^1$, for any word $\eta\in X_\omega ^*:= \bigcup_N X_N ^{(\omega)}$ and any $x\in f_\eta(K)\cap K_\omega$, there is some $y\in f_\eta(K)\cap K_\omega$ such that 
$$\left| \left\langle x-y,\,w \right\rangle \right| > \delta\cdot \left| f_\eta ' (0) \right|.$$

\item  For all $D>1$ we can find a constant $C_D=C_D(\Sigma)>0$ such that:

For all $\omega\in \Omega$ and $x\in \supp (\mu_\omega)$, for every $r>0$,
$$\mu_{\omega}  \left( B(x, D\cdot r) \right) \leq C_D  \mu_{\omega}  \left( B(x, r) \right).$$
That is, every $\mu_\omega$ satisfies the doubling (Federer) property.
\end{enumerate}
\end{theorem}
Recall that by Claim \ref{Claim key} our IFS $\Phi$ satisfies the assumptions of Theorem \ref{Theorem disint}. The main difference between Theorem \ref{Theorem disint} and its one-dimensional counterpart \cite[Theorem 2.4]{algom2023polynomial}, apart from the different transfer operator, is Part (6). This is a version of Oh-Winter's non-concentration property \cite[Section 4.1]{Oh2017winter}. Our  version is obviously different due to the model and self-conformal  setting. But, it does serve the same moral purpose - none of the $K_\omega$ look locally like circles as we zoom into any point in them. As for the proof, it differs from that of \cite[Theorem 2.4]{algom2023polynomial} only in the way we construct the model, and in Part (6). The other Parts of the proof are similar to their one-dimensional counter-parts, and we invite the reader to consult with that proof.

\subsubsection{Constructing the model}
Referring to  conditions \eqref{Add property 1} and \eqref{Add property 2} and their notations (recall that we assume they hold true), for $1\leq i\leq 4$ we define IFSs
$$\Psi_1  = \Psi_2= \lbrace f_1, f_2 \rbrace, \text{ and put } \Psi_3=\Psi_4= \lbrace f_3, f_4 \rbrace.$$
For all $k\notin \lbrace 1,...,4\rbrace$ we define the IFS 
$$\Psi_k  := \lbrace f_i,f_{i+1},f_k \rbrace$$ 
for the $i$ that is given by \eqref{Add property 1}. Note that by \eqref{Add property 1}, each $\Psi_k$ satisfies that
$$\bigcup_{f\in \Psi_k} f(D) \text{ is a disjoint union}.$$
Thus, for every $k\in \mathcal{A}$ we define $\Phi_k$ as a finite set of maps that  has the maximal possible cardinality among those satisfying
\begin{equation}\label{Eq max card}
\Psi_k \subseteq \Phi_k \subseteq \Phi, \text{ and } \bigcup_{f\in \Phi_k} f(D) \text{ is a disjoint union}.
\end{equation}
Note that such a collection may not be unique - in this case, simply pick one such family of maps. Thus, we put
$$I=\lbrace  \Phi_j:\, j\in \mathcal{A} \rbrace.$$
In particular, every $i\in \mathcal{A}$ corresponds to the IFS $\Phi_i$. 

Note that there is repetition in this construction: there are $f_j\in \Phi$ that appear in several IFSs in $I$. So, for every $i\in \mathcal{A}$ we put
$$n_i := \left| \lbrace j:\, f_i \in \Phi_j \rbrace \right|.$$
By definition $\nu=\nu_{\mathbf{p}}$ for a probability vector $\mathbf{p}\in \mathcal{P}(\mathcal{A})$. So, we can define $\mathbf{q}$, a probability vector on $I$, via  
$${\mathbf{q}}_j = \sum_{i\in \mathcal{A}:\, f_i \in \Phi_j} \frac{p_i}{n_i}\quad \text{ for all } j\in I.  $$
We can now define our Bernoulli selection  measure on $\Omega$ as $\mathbb{Q}:=\mathbf{q}^\mathbb{N}$. 

Finally, for all $i\in I$ we define the strictly positive probability vector
$$\tilde{\mathbf{p}}_j (i) = \frac{\frac{p_i}{n_i}}{\mathbf{q}_j}, \text{ for every } f_i \in \Phi_j. $$
This completes the construction of our model.

\subsubsection{Proof of Part (6) of Theorem \ref{Theorem disint}}
We first prove that the set $K$ has the non concentration property, similarly to \cite[Definition 4.2]{Oh2017winter}. Morally, the following Lemma means that if $K$  does not lie in an analytic curve (which is one of our assumptions), then every microset of $K$ does not lie in an analytic curve.
\begin{Lemma} \label{Lemma delta}
Suppose $K$ is not contained in an analytic curve. Then there exists $\delta'>0$ such that:

For every $x\in K, w\in S^1, \eta \in \mathcal{A}^*$ such that $x\in f_\eta (K)$, there exists $y\in f_\eta (K)$ such that
$$\left| \left\langle x-y,\,w \right\rangle \right| > \delta\cdot \left| f_\eta ' (0) \right|.$$
\end{Lemma}
\begin{proof}
We follow the proof of \cite[Theorem 4.3]{Oh2017winter}. If the conclusion fails, then there are $x_n \in K, w_n \in S^1,\eta_n\in \mathcal{A}^*$ such that $x_n \in f_{\eta_n} (K)$ and  for any sequence $y_n \in f_{\eta_n} (K)$
$$\frac{\left| \left\langle x_n-y_n,\,w_n \right\rangle \right|}{\left| f_{\eta_n} ' (0) \right|}\rightarrow 0.$$
Note that we may assume $|\eta_n|\rightarrow \infty$. Indeed, otherwise we may assume $\eta_n = \eta$ is fixed. Taking convergent subsequences of $x_n \rightarrow x$ and $w_n \rightarrow w$, and since $f_\eta (K)$ is perfect, it follows that  for every $y\in f_\eta (K)$
$$\left\langle x-y,\,w \right\rangle=0.$$
Therefore, $f_\eta (K)$ is contained in a line. Thus, $K$ is contained in an analytic curve, a contradiction.

Thus, writing $\epsilon_n :=\left| f_{\eta_n} ' (0) \right|$, we may assume $\epsilon_n \rightarrow 0$.  Put
$$\phi_n(y):= \frac{y-x_n}{\epsilon_n}.$$
Consider the map
$$\phi_n \circ f_{\eta_n}:D\rightarrow D.$$
By bounded distortion the derivative of $\phi_n \circ f_{\eta_n}$ is uniformly bounded, so we may assume it converges locally uniformly to a non-constant univalent function $g:D\rightarrow D$. Using the same convergent subsequences $x_n \rightarrow x$ and $w_n \rightarrow w$ as above, it follows that $K$ is contained in the curve
$$g^{-1} \left( (w-x)^\perp \right).$$
This is a contradiction. 

\end{proof}

Note that $\delta'$  as in Lemma \ref{Lemma delta} depends only on $K$. Also, the constant $L$ from the bounded distortion principle \eqref{Eq bdd distortion} is the same regardless of inducing. So, by further inducing if needed, we may assume that
$$\delta'-2L\cdot \rho >0.$$
Indeed, here we are using that the only quantity that changes in the above equation when we induce the IFS is $\rho$, that becomes increasingly small. Also, note that if needed we can assume our original induced IFS already satisfies this condition (otherwise, we first induce and then run the construction as  in Claim \ref{Claim key}).

The moral meaning of the following Lemma is that if $K$  does not lie in an analytic curve, and we are careful about constructing our model, then for every $\omega \in \Omega$, every microset of $K_\omega \cap K$  does not lie in an analytic curve. We note that this property can fail in general: For example, any model constructed by using at least one degenerate IFS $\lbrace f \rbrace$ would fail this property (indeed, here at least one $K_\omega$ will just be a point). 
\begin{Lemma} \label{Lemma iterative}
There exists $\delta>0$ such that: For every $\omega \in \Omega$, $x\in K_\omega$, $w\in S^1, \eta\in X_\omega ^*$ such that $x\in f_\eta(K)\cap K_\omega$, there is some $y\in f_\eta(K)\cap K_\omega$, such that
$$\left| \left\langle x-y,\,w \right\rangle \right| > \delta \cdot \left| f_\eta ' (0) \right|.$$
\end{Lemma}
\begin{proof}
By Lemma \ref{Lemma delta}  there exists $y\in f_\eta (K)$ such that the conclusion holds with $\delta'$. Our goal is  show that we may assume $y\in K_\omega$ as well, by perhaps replacing $\delta'$ with a smaller $\delta$.

Thus, if $y\in K_\omega$ we are done. Otherwise, let $j\in \mathcal{A}$ be such that $y\in f_{\eta *j}(K)$. 
\begin{itemize}
\item Suppose $\eta j\in X_\omega ^*$. Since $D$ is convex,  for any $z\in f_{\eta j}(K)\cap K_\omega$ there is some $z_0\in D$ such that 
$$\left| y-z\right|\leq \left| f_{\eta j} ' (z_0)\right| \leq L\cdot \rho \cdot \left| f_\eta ' (0)\right|.$$
Thus,
$$\left| \left\langle x-z,\,w \right\rangle \right| \geq \delta'\cdot \left| f_\eta ' (0) \right|-L\cdot \rho \cdot \left| f_\eta ' (0)\right|.$$
So, $z\in f_\eta(K)\cap K_\omega$ does the job with $\delta = \delta'-L\cdot \rho>0$.

\item  Suppose $\eta*j\notin X_\omega ^*$. Then there must exist some $i\in \Phi_{\omega_{|\eta|+1}}$ such that $f_i(D)\cap f_j(D)\neq \emptyset$. Indeed, otherwise the IFS $\Phi_{\omega_{|\eta|+1}}$ contradicts the maximum-cardinality condition we imposed in \eqref{Eq max card}.

Then for any $z\in f_{\eta*i}(D)$, arguing as above
$$\left| y-z\right|\leq  2L\cdot \rho \cdot \left| f_\eta ' (0)\right|.$$
Thus,
$$\left| \left\langle x-z,\,w \right\rangle \right| \geq \delta\cdot \left| f_\eta ' (0) \right|-2L\cdot \rho \cdot \left| f_\eta ' (0)\right|.$$
So,  any $z\in f_{\eta*i}(K)\cap K_\omega$ does the job with $\delta = \delta'-2L\cdot \rho>0$.
\end{itemize}

\end{proof}

\subsection{Statement of spectral gap and reduction to the construction of Dolgopyat operators} \label{Section reduction}
A standard norm to put on $C^1(D)=C^1(D,\, \mathbb{C})$ is
\begin{equation} \label{Eq B-Q norm}
||\varphi||_{C^1} = ||\varphi||_\infty + ||\nabla \varphi||_{\infty}.
\end{equation}
As in the works of Oh-Winter \cite[equation (2.2)]{Oh2017winter} and  Dolgopyat \cite[Section 6]{Dolgopyat1998rapid},  we  put  another norm on $C^1(D)$ via, for every $b \neq 0$,
\begin{equation} \label{Eq norm dol theta}
||\varphi||_{(b)} = ||\varphi||_\infty + \frac{||\nabla \varphi||_{\infty}}{|b|} \text{ if } b\geq 1, \, ||\varphi||_{(b)}=||\varphi||_{C^1} \text{ otherwise}.
\end{equation}
It is clear that these norms are equivalent.

Let $\Phi$ be an IFS as in Theorem \ref{Theorem disint}, and let $\nu=\nu_\mathbf{p}$ be a self-conformal measure. We can now state the main result of this Section, and also the main technical result of this paper, which is that the operators $P_{s,\ell}$ have spectral gap in some uniform sense. This is a stripped down axiomatic version of a spectral gap result of Oh-Winter \cite[Theorem 2.5]{Oh2017winter} for certain Julia sets. In particular, we do not require a Markov partition, or that $\nu$ is a doubling measure. It is also a higher dimensional version of our previous work \cite[Theorem 2.8]{algom2023polynomial}.
\begin{theorem} \label{Theorem spectral gap}
Let $\Phi$ be a  $C^\omega (\mathbb{C})$ IFS satisfying \eqref{Add property 1} and \eqref{Add property 2}. In addition, assume $K_\Phi$ does not belong to an analytic curve.

Then for any self-conformal measure $\nu=\nu_{\mathbf{p}}$ there are $C,\gamma, \epsilon,R>0$ and $0<\alpha<1$ such that for all $|b|+|\ell|>R$, $a\in \mathbb{R}$ satisfying $|a|<\epsilon$, and integer $n>0$
\begin{equation} \label{Eq spectral gap}
||P_{a+ib, \ell} ^{n} ||_{C^1} \leq C\cdot \left( |b|+|\ell|\right)^{1+\gamma}\cdot  \alpha^n.
\end{equation}
\end{theorem}
By the discussion in Section \ref{Section induced IFS}, any IFS satisfying the assumptions of Theorem \ref{Main Theorem analytic} has an induced IFS that meets the condition of Theorem \ref{Theorem spectral gap}.

It is a standard fact that Theorem \ref{Theorem spectral gap} may be reduced to the construction of Dolgopyat operators, that we now discuss; this follows by an $L^2$-reduction argument, see the end of this subsection. Given $A>0$ we define the cone
$$\mathcal{C}_A = \lbrace f\in C^1(D): \, f>0,\text{ and } \left|\nabla f(x)\right|\leq A\cdot f(x), \, \forall x\in D \rbrace.$$
We remark that for any $u,v \in D$ and $f\in \mathcal{C}_A$ we have
$$e^{-A|u-v|} \leq f(u)/f(v) \leq e^{A|u-v|}.$$
The following Lemma is the key to the proof of Theorem \ref{Theorem spectral gap}; in fact, it implies it by standard methods. It is a higher dimensional  variant of \cite[Lemma 2.12]{algom2023polynomial}. We remark that we differ a bit from how Oh-Winter organize their proof; the closest analogue in their work is \cite[Theorem 5.7]{Oh2017winter}. In the following statement we refer to the model constructed in Theorem \ref{Theorem disint}, with its various notations (it has the same assumptions as Theorem \ref{Theorem spectral gap}). See Section \ref{Section dis}.
\begin{Lemma} \label{Lemma 5.4 Naud}
Assume the conditions of Theorem \ref{Theorem spectral gap} are met, and let $\Sigma$ be the model from Theorem \ref{Theorem disint}. 
There is an integer $N>0$, some $A>1$ and some $0< \alpha <1$ such that for every $s=a+ib, \ell\in \mathbb{Z}$ where $|a|$ is small and $|b|+|\ell|$ is large:

For all $\omega\in \Omega$ there is a finite set of bounded operators $\lbrace N_{s,\ell} ^J \rbrace_{J\in \mathcal{E}_{s, \ell, \omega}}$ on the space $C^1(D)$ satisfying:
\begin{enumerate}
\item For every $J\in \mathcal{E}_{s, \ell, \omega}$ the operator $N_{s,\ell} ^J$ stabilizes the cone $\mathcal{C}_{A\left( |b|+ |\ell| \right)}$.

\item If $H\in \mathcal{C}_{A\left( |b|+ |\ell| \right)}$ and $J\in \mathcal{E}_{s,\ell, \omega}$ then
$$\int | N_{s,\ell} ^J H|^2 d \mu_{\sigma^N \omega} \leq \alpha\cdot \int P_{0,0, \omega,N} \left( \left| H^2 \right| \right) d \mu_{\sigma^N \omega}.$$

\item Suppose $H\in \mathcal{C}_{A\left( |b|+ |\ell| \right)}$ and $f\in C^1(D)$ are such that $|f|\leq H$ and $|\nabla f|\leq A\left( |b|+ |\ell| \right)H$. Then for all $\omega \in \Omega$ we can find some $J\in \mathcal{E}_{s,\ell, \omega}$ satisfying
$$|P_{s, \ell, \omega, N} f| \leq N_{s,\ell} ^J H \quad \text{ and } |\nabla (P_{s, \ell, \omega, N} f)|\leq A\left( |b|+ |\ell| \right)N_s ^J H.$$
\end{enumerate} 
\end{Lemma}
The operators  we construct in Lemma \ref{Lemma 5.4 Naud} are called \textit{Dolgopyat operators}.

Let us now briefly explain why Lemma \ref{Lemma 5.4 Naud} implies Theorem \ref{Theorem spectral gap}. The idea is that Lemma \ref{Lemma 5.4 Naud} implies the following $L^2$ contraction statement in Proposition \ref{Prop. 5.3 Naud}. This Proposition, in turn, is known to imply Theorem \ref{Theorem spectral gap}: The proof that Proposition \ref{Prop. 5.3 Naud} implies Theorem \ref{Theorem spectral gap} is similar to the corresponding argument in \cite[Section 2.3]{algom2023polynomial}. The proof that Lemma \ref{Lemma 5.4 Naud} implies Proposition \ref{Prop. 5.3 Naud} is also similar to the one in \cite[Section 2.4]{algom2023polynomial}. We thus omit the details.

\begin{Proposition} \label{Prop. 5.3 Naud} Suppose the assumptions of Theorem \ref{Theorem spectral gap} are met.  There exist $N\in \mathbb{N}$ and some  $0<\alpha<1$ so that for $s=a+ib$ where $|a|$ is small enough and $|b|+|\ell|$ is large enough that satisfy: 

For all $\omega \in \Omega$,
$$\int \left| P_{s, \ell,  \omega , nN} W \right|^2 \, d\mu_{\sigma^{nN} \omega} \leq \alpha^n$$
for every $W\in C^1(D)$ such that $||W||_{(|b|+|\ell|)} \leq 1$. 
\end{Proposition}
We note that this is a randomized analogue of \cite[Theorem 2.7]{Oh2017winter} in Oh-Winter, the proof of which they attribute to Ruelle.

\subsection{Proof of Lemma \ref{Lemma 5.4 Naud}}  \label{Section proof}
\subsubsection{Some preliminary Lemmas}
In this Section we collect a few results that will be used in the proof of Lemma \ref{Lemma 5.4 Naud}. Fix the model given by Theorem \ref{Theorem disint}, and let us use freely its notations. First, we require the following standard covering result:
\begin{Lemma} \label{Lemma Vitali}
For every $\omega \in \Omega$ and every $\epsilon>0$ there exist $q=q(\omega, \epsilon)\in \mathbb{N}$ points $x_1,...,x_q\in K_\omega$ such that:
\begin{enumerate}
\item $K_\omega \subseteq \bigcup_{i=1} ^q B_{50 \epsilon} (x_i)$.

\item If $i\neq j$ then $B_{10 \epsilon} (x_i) \cap B_{10 \epsilon} (x_j)=\emptyset.$
\end{enumerate}
\end{Lemma}
\begin{proof}
This follows a Vitali covering argument, see e.g. \cite[Section 1.3]{falconer1986geometry} or \cite[Section 5.3]{Oh2017winter}.
\end{proof}

We also record the following useful and elementary Lemma. We use the notation $D'$ given by \eqref{Eq C and C prime}. 
\begin{Lemma} (A-priori bounds) \label{Lemma iterations}
For some $C_1>0$, for all $|a|$ small enough and $|b|+|\ell|$ large enough, for every $f\in C^1 (D)$, if $s=a+ib$, for all integer $n \in \mathbb{N},\ell,$ we have
$$|| \nabla \left( P_{s, \ell} ^n f\right)||_\infty \leq C_1 \left( |b|+|\ell|\right) \cdot ||P_a ^n f||_\infty+\rho ^n || P_a ^n \left( |\nabla f| \right)||_\infty, \text{ and also}$$
$$||\nabla  \left( P_{s, \ell, \omega, n} f\right) ||_\infty \leq C_1 \left( |b|+|\ell| \right) \cdot ||P_{a,  \omega, n}  f||_\infty+\rho ^n || P_{a, \omega, n} \left( |\nabla f| \right)||_\infty, \, \text{ for every }\omega.$$
Therefore, for some $C_6>0$, for every two integer $n\in \mathbb{N}, \ell$ and $s=a+ib$
$$|| P_{s,\ell} ^{n} ||_{|b|+|\ell|} \leq C_6 \cdot e^{nD' |a|}.$$
\end{Lemma}
This is similar to \cite[Lemma 2.11]{algom2023polynomial} and \cite[Lemma 5.1]{Oh2017winter}, so we omit the details.

Next, we define cylinders inside our model construction:
\begin{Definition} \label{Def cylinders}
Let  $n\in \mathbb{N}$ and $\omega\in \Omega$. For every $u \in X_n ^{(\omega)}$ we define the corresponding cylinder via
$$C_u := f^{ (\omega_1)} _{u_1} \circ \dots \circ f^{ (\omega_n)} _{u_n} (D) \subseteq D.$$
\end{Definition}
Let us now fix $\omega \in \Omega$ for the subsequent discussion. We will be careful though to derive estimates that are uniform across all $\omega \in \Omega$. Now, note that
\begin{equation} \label{Eq measure of cylinder}
\mu_\omega (C_u) = \mu_\omega ( f^{ (\omega_1)} _{u_1} \circ \dots \circ f^{ (\omega_n)} _{u_n} (D) ) = \eta^{(\omega)} ([u_1,...,u_n]) = \prod_{i=1} ^n \mathbf{p}^{(\omega_i)} _{u_i}.
\end{equation}
Indeed, this follows Theorem \ref{Theorem disint} Part (4),  the definition of $\mu_\omega$, and by \eqref{Eq stochastic self-similarity}.

We denote the diameter of a cylinder set by $|C_\alpha|$.  The following  is similar to \cite[Lemma 6.1]{Naud2005exp}:
\begin{Lemma} \label{Lemma 6.1 Naud}
For some $C>0$ and some $0<r_1,r_2<1$ that are uniform in $\omega$, for every two nesting cylinders $C_\alpha \subseteq C_\beta$ we have
$$C^{-1} r_1 ^{|\alpha|-|\beta|} \leq |C_\alpha|/|C_\beta| \leq C\cdot r_2 ^{|\alpha|-|\beta|}.$$
\end{Lemma}
\begin{proof}
Put $\alpha=\beta* u$ assuming $|u|=|\alpha|-|\beta|$. Then,  by \eqref{MVT} for some $x_1 \in D$
$$\left| C_\alpha \right| = |f_\beta \circ f_u (D)| \leq  \left|f_\beta ' \circ f_u(x_1) \cdot f_u' (x_1)\right|\cdot 2.$$
Similarly, by \eqref{MVT},  for some $y_1$ we have
$$\left| C_\beta \right| = |f_\beta(D)| \leq  \left|f_\beta ' (y_1) \right|\cdot 2.$$
In addition, by Cauchy's integral formula for derivatives,
$$ \left| f_\alpha'(0) \right|\leq \left| C_\alpha \right|,\, \text{ and }  \left| f_\beta'(0) \right|\leq \left| C_\beta \right|.$$ 
Indeed, invoking Cauchy's formula for the map $g(x)=f_\alpha(x)-f_\alpha(0)$ we obtain
$$f_\alpha'(0) = g'(0)= \frac{1}{2\pi i} \int_{\partial D} \frac{g(z)}{z^2}\,dz.$$
So,
$$\left| f_\alpha'(0) \right|  \leq \frac{1}{2\pi} \cdot \text{length}(\partial D)\cdot \sup_{z\in \partial D} |g(z)| \leq |C_\alpha|.$$
Note the use of the relation $g(x)=f_\alpha(x)-f_\alpha(0)$ in the last inequality.

Recall the definition of $L$ from \eqref{Eq bdd distortion}. 
So, since $f_\beta$ is a composition of functions from $\Phi$,
$$\frac{1}{2}\cdot  L^{-1} \left( \min_{i\in I, f \in \Phi^{(i)}} \min_{x\in D} \left| f_i ' (x)\right| \right)^{|u|} \leq \frac{1}{2}\cdot  L^{-1} \cdot \left|f_u' (0)\right| \leq  \frac{|C_\alpha|}{|C_\beta|} \leq  2\cdot L \cdot \left|f_u' (x_1)\right|\leq  2\cdot L\cdot \rho^{|u|},$$
as claimed.
\end{proof}

We also require the following Lemma.
\begin{Lemma} \label{Lemma 7.1 Nuad}
There is some uniform constant $B_1$ so that for every $x\in K_\omega$ and every $r>0$, there is a cylinder $C_\alpha = f_\alpha(D)$ so that
$$C_\alpha \subseteq B_r(x) \text{ and we have } B_1  r \leq |C_\alpha|.$$
\end{Lemma}
\begin{proof}
Let $C_\alpha$ be  a cylinder such that $x\in C_\alpha\subseteq B_r(x)$, where $|\alpha|$ is assumed to be minimal. Since $x\in K_\omega$ such a cylinder does indeed exist. Now put $C_\alpha \subseteq C_\beta$ where $|\beta|=|\alpha|-1$. By the minimality of $|\alpha|$ it is impossible that $C_\beta$ is a subset of $B_r(x)$, and thus $\left| C_\beta \right| \geq r$. Invoking Lemma \ref{Lemma 6.1 Naud} we thus see that
$$C^{-1} r_1 r \leq C^{-1} r_1 |C_\beta| \leq |C_\alpha|$$
as claimed.  
\end{proof}

Finally, we require the following Lemma, which is a version of \cite[Lemma 3.2]{Oh2017winter}. It is a reformulation on the (UNI) condition.  For a $C^1$ map $h$ regarded as a map $\mathbb{R}^2 \rightarrow \mathbb{R}^2$ we define its $C^1$ norm $\vert | h \vert|_{C^1}$ as $||h||_\infty +  || \nabla h(u) ||_\infty$, usually restricted to $D$. We define its $C^2$ norm as the maximum of its $C^1$ norm and the supremum of $X^2 h$ as $X$ ranges over all unit tangent vectors.

\begin{Lemma}  (Reformulation of UNI) \label{Lemma 3.2 Oh}
Let $m',m>0$ and $N_0\geq 0$ be such that for all $N\geq N_0$, for every $\omega \in \Omega$ there exist $\alpha_1 ^N, \alpha_2 ^N \in X_N ^{(\omega)}$ such that
 
\begin{equation} \label{Eq condition}
m \leq \left|\nabla  \left(  \log \left| f_{\alpha_1 ^N } ' \right| - \log \left| f_{\alpha_2 ^N } ' \right| \right) \left(x\right) \right| \leq m', \quad \text{ for all } x\in D.
\end{equation}
Then there exists some $\delta_2>0$ such that for all $N\geq N_0$, the map
\begin{equation} \label{Eq 3.3 Naud}
T_N:= \left(  \left(  \log \left| f_{\alpha_1 ^N } ' \right| - \log \left| f_{\alpha_2 ^N } ' \right| \right) ,\,  \left(  \arg f_{\alpha_1 ^N } '  - \arg  f_{\alpha_2 ^N } '  \right) \right):D \rightarrow \mathbb{R}\times \mathbb{T}
\end{equation}
is a local diffeomorphism satisfying $\vert \vert T_N \vert \vert_{C^2} < \frac{1}{2 \delta_2}$ and
$$ \inf_{ u \in D} \left|  \nabla T_N (u)  \cdot v \right| \geq \delta_2 \cdot |v|,\quad \forall v\in \mathbb{R}^2.$$
\end{Lemma}
Note that the existence of the parameters as in \eqref{Eq condition} follows from Claim \ref{Claim UNI}.
\begin{proof}
Let $\alpha_1 ^N, \alpha_2 ^N \in X_N ^{(\omega)}$ be as in the statement of the Lemma. Then  $h= f_{\alpha_1 ^N} '/f_{\alpha_2 ^N} '$  is holomorphic. Furthermore, we can write 
$$h =\exp \left(  \log \left| f_{\alpha_1 ^N } ' \right| - \log \left| f_{\alpha_2 ^N } ' \right| + i\left(  \arg f_{\alpha_1 ^N } '  - \arg  f_{\alpha_2 ^N } '  \right) \right).$$
By \eqref{Eq condition}, 
$$\left( \log h \right)'\neq 0 \text{ on } D.$$
 Thus, for every $x\in D$ the map $\log h$ defines a local diffeomorphism from a neighbourhood of $x$ to some open subset of $\mathbb{C}$. Furthermore, we have global estimates about the Jacobian of $T_N$  since $\log h$ is holomorphic  and so
$$ \det  \nabla T_N   = \left| \frac{d}{dx} \log h \right|^2 =  \left|\nabla  \left(  \log \left| f_{\alpha_1 ^N } ' \right| - \log \left| f_{\alpha_2 ^N } ' \right| \right) \left(x\right) \right|^2.$$ 
We can now simply apply \eqref{Eq condition}.

As for the $C^2$ norm, we use a compactness argument: Let us denote the function $h$ as above by $h_{N,\omega}$. Let
$$s_\omega := \sup_{N} ||\log h_{N,\omega}||_{C^2}.$$
Then there exists a subsequence $N_k$ such that $||\log h_{N_k,\omega} ||_{C^2}\rightarrow s_\omega$. Note that\footnote{Here, if needed, we can also move to a smaller open neighbourhood of the attractor in $D$. Alternatively, as pointed out by the referee,  one can argue more directly by replacing $h_{N,\omega}$ with $h_{N,\omega}/h_{N,\omega}(x_0)$ for some fixed $x_0$.} $s_\omega <\infty$ for all $\omega$  since the family of homomorphic functions $\log h_{N_k,\omega}$ is pre-compact by \eqref{Eq condition}. This also implies that there exists a holomorphic function $H_\omega$  such that
\begin{equation} \label{Eq precompact}
||H_\omega ||_{C^2} = s_\omega,\quad \text{ and } \sqrt{m} \leq |H_\omega ' (x)|\leq \sqrt{m'}.
\end{equation} 
Finally, let
$$S = \sup_\omega s_\omega.$$
Then there exists a subsequence $N_\ell$ such that $s_{\omega_\ell} \rightarrow S.$ By \eqref{Eq precompact} there is a further subsequence $N_{\ell_k}$ such that $||H_{\omega_{N_{\ell_k}}} ||_{C^2} \rightarrow S$. Since the family $\lbrace H_\omega \rbrace_{\omega \in \Omega}$ is precompact by \eqref{Eq precompact}, there exists a holomorphic function $H$ such that 
$$||H||_{C^2}= S.$$
Unfolding the definition of $S$ we get the assertion about the $C^2$ norm, and the Lemma follows.  
\end{proof}

\subsubsection{Construction of the Dolgopyat operators}\label{Section construction}
Fix $\omega$ and let $s=a+ib\in \mathbb{C}$ and $\ell \in \mathbb{Z}$. In this subsection we construct the Dolgopyat operators from Lemma \ref{Lemma 5.4 Naud}. Our construction parallels that of Oh-Winter \cite[Section 5.3]{Oh2017winter}, but is significantly different due to our model self-conformal setting. Let $N\in \mathbb{N}$ be sufficiently large in the sense of Lemma \ref{Lemma 3.2 Oh}; more conditions on the minimal length of $N$ will be given later. Let $\alpha_1 ^N, \alpha_2 ^N \in X_N ^{(\omega)}$ be  length $N$ words such that the outcome of Lemma \ref{Lemma 3.2 Oh} is satisfied.

Let us fix $\epsilon_1 ,\frac{1}{ |b|+|\ell|}>0$  uniformly in $\omega$, that are small  (to be determined later). Let 
$$(B_i = B_{\tilde{\epsilon}} (x_i))_{1\leq i\leq q},\quad x_i\in K_{\sigma^N \omega}$$
be a Vitali covering  as in Lemma \ref{Lemma Vitali} of $K_{\sigma^N \omega}$ of modulus
$$\tilde{\epsilon}=\frac{\epsilon_1}{|b|+|\ell|}.$$
 For all $(i,j)\in \lbrace 1,2\rbrace \times \lbrace 1,...,q\rbrace$ set
\begin{equation} \label{Eq wi and wi hat}
w_i = b\cdot \nabla \left(  \log \left| f_{\alpha_1 ^N } ' \right| - \log \left| f_{\alpha_2 ^N } ' \right| \right)(x_i)+\ell\cdot \nabla \left( \arg f_{\alpha_1 ^N } '  - \arg  f_{\alpha_2 ^N } '  \right)(x_i),\quad \text{ and } \hat{w_i}=\frac{w_i}{|w_i|}.
\end{equation}
By Lemma \ref{Lemma 3.2 Oh}, assuming $\tilde{\epsilon}$ is small enough,
\begin{equation} \label{Eq 5.7 Oh}
|w_i|> \frac{\delta_2\cdot \left(|b|+|\ell|\right) }{2}.
\end{equation}
Let $\delta_1$ be the $\delta$ given by Theorem \ref{Theorem disint} part (6). Then by this result and Lemma \ref{Lemma 7.1 Nuad} (which is where $B_1$ comes from), for each $i$ there exists a point $y_i \in B_{ 5\tilde{\epsilon}} (x_i) \cap K_{\sigma^N \omega}$ such that 
\begin{equation} \label{Eq 5.7 oh 2}
\left| x_i-y_i \right| \geq \left| \left\langle x_i-y_i,\,\hat{w_i} \right\rangle \right| > \delta_1 \cdot L\cdot B_1 \cdot 5\tilde{\epsilon}.
\end{equation}
We impose a first restriction on $\epsilon_1$, that $\epsilon_1\cdot  L \cdot  B_1 <1/2$.

For each $x\in \mathbb{C}$ and $\epsilon>0$ we pick a smooth bump function $\psi_{x,\epsilon}\in C^\infty(\mathbb{C},\, \mathbb{R})$ such that:
\begin{enumerate}
\item $\psi_{x,\epsilon} \equiv 0$ on $\mathbb{C}\setminus B_\epsilon(x)$.

\item $\psi_{x,\epsilon} \equiv 1$ on $B_{\epsilon/2} (x)$.

\item $||\psi_{x,\epsilon}||_{C^1} \leq  \frac{4}{\epsilon}.$
\end{enumerate}

We can now define our parameter space as
$$J_{s, \ell, \omega} := \lbrace 1,2\rbrace \times \lbrace 1,2\rbrace \times \lbrace 1,...,q\rbrace.$$
We identify $(i,k,j) \in J_{s, \ell, \omega}$ as follows: The first coordinate $i$ will correspond to $f_{\alpha_i ^N}$. The last coordinate is some $j\in \lbrace 1,...,q\rbrace$. The second coordinate $k$ corresponds to $x_j$ if $k=1$ and to $y_j$ if $k=2$. This will help clarify the following definition.

Fix $0< \theta <1$, $A>1$ that will be determined along the proof. For any $\emptyset \neq J \subseteq J_{s, \ell, \omega}$ we define a function $\chi_J$ on $D$ via, for $\delta_3:= \frac{\delta_1 \cdot  \delta_2}{4 A}$, 
\[
  \chi_J =
  \begin{cases}
    1- \theta \cdot \left(\sum_{(1,1,j)\in J} \psi_{x_j, 2 \delta_3 \tilde{\epsilon} } \circ f_{\alpha_1 ^N} ^{-1} + \sum_{(1,2,j)\in J} \psi_{y_j, 2 \delta_3 \tilde{\epsilon} } \circ f_{\alpha_1 ^N} ^{-1}  \right)  & \text{on $f_{\alpha_1 ^N}(D)$} \\
    1- \theta \cdot \left(\sum_{(2,1,j)\in J} \psi_{x_j, 2 \delta_3\tilde{\epsilon} } \circ f_{\alpha_2 ^N} ^{-1} + \sum_{(2,2,j)\in J} \psi_{y_j, 2 \delta_3\tilde{\epsilon} } \circ f_{\alpha_2 ^N} ^{-1}  \right)  & \text{on $f_{\alpha_2 ^N}(D)$} \\
    1 & \text{otherwise}
  \end{cases}
\]
Note that by Theorem \ref{Theorem disint} Part (4) this is well defined.

Our Dolgopyat operators $N_{s,\ell} ^J$ on $C^1 (D)$ are defined as
$$N_{s,\ell} ^J \left( f \right)(x):= P_{a,0,\omega, N} \left( \chi_J \cdot f \right), \, J\in J_{s, \ell, \omega}.$$
In the next three subsection we define our parameters, and prove the three parts of Lemma \ref{Lemma 5.4 Naud}.
\subsubsection{Proof of Part 1 of Lemma \ref{Lemma 5.4 Naud}}
Let us retain the notations of the previous section. Here, we prove Lemma \ref{Lemma 5.4 Naud} Part (1). Since we have Lemma \ref{Lemma iterations} at our disposal, the proof of the following Lemma is very similar to \cite[Lemma 2.17]{algom2023polynomial}. However, we give full details here since we need to indicate how to pick the  parameters for Lemma \ref{Lemma 5.4 Naud}. We remark that \cite[Theorem 5.6 part (1)]{Oh2017winter} is the corresponding result in the work of Oh-Winter.
\begin{Lemma} \label{Lemma cone}
There is some $A>1$, an integer $N>0$, and some $0<\theta <1$ such that for all $s=a+ib$ where $|a|$ is sufficiently small and $|b|+|\ell|$ is sufficiently large, for all $\omega$,  
\begin{enumerate}
\item The operator $N_{s,\ell} ^J$ leaves the cone $\mathcal{C}_{A(|b|+|\ell|)}$ invariant.

\item For every $f\in C^1(D)$ and $H\in \mathcal{C}_{A(|b|+|\ell|)}$ such that
$$ |f|\leq H \text{ and } \left|\nabla f \right| \leq A(|b|+|\ell|)H$$
we have
$$\left| \nabla \left( P_{s,\ell, \omega,N} \left( f \right) \right) (x) \right| \leq  A(|b|+|\ell|) N_{s,\ell} ^J (H) (x).$$

\item For every $H\in \mathcal{C}_{A(|b|+|\ell|)}$ we have $P_{a,0,\omega,N} \left(|H|^2 \right) \in \mathcal{C}_{ \frac{3}{4}A(|b|+|\ell|)}$.
\end{enumerate}
\end{Lemma}

\begin{proof}
Let $H\in \mathcal{C}_{A(|b|+|\ell|)}$, where conditions on $A$ will be given later. For every $x\in D$,
$$\left| \nabla N_{s, \ell} ^J \left( H \right) (x) \right| = \left| \nabla P_{a,0,\omega,N} \left( \chi_J \cdot H \right)(x) \right|$$
$$ \leq \sum_{I\in X_N ^{(\omega)}} e^{a\cdot 2 \pi \cdot  c(I,x)} \left| a 2 \pi\cdot \nabla \left( \log f_I ' \right)(x) \right|  \eta^{(\omega)}(I) \left(H\cdot \chi_J \right)\circ f_I (x)$$
$$ + e^{a\cdot 2 \pi \cdot  c(I,x)} \eta^{(\omega)}(I)  \left|\nabla  \left( H\cdot \chi_J  \right) \circ f_I  (x) \right|.$$
By Theorem \ref{Theorem disint} Part (4) for all $I\in X_N ^{(\omega)}$ 
$$\left| \nabla \left(\chi_J  \circ f_{I} \right) \right| \leq 4\cdot \theta\cdot  \frac{|b|+|\ell|}{\epsilon'} \cdot 2\delta_3.$$
By \eqref{Eq tilde C} we can assume $\tilde{C}$ satisfies 
$$\left| a 2 \pi\cdot \nabla \left( \log |f_{I} '| \right) (x) \right|  \leq \tilde{C},$$
uniformly in $N$, and assuming $|a|$ is small enough.
Thus,
$$\left| \nabla N_{s, \ell} ^J \left( H \right) (x) \right| \leq \sum_{I\in X_N ^{(\omega)}} e^{a\cdot 2 \pi \cdot  c(I,x)}  \eta^{(\omega)}(I)) \left(H\cdot \chi_J  \right)\circ f_I (x) \cdot \tilde{C}$$
$$  + e^{a\cdot 2 \pi \cdot  c(I,x)} \eta^{(\omega)}(I)  \left(H\cdot \chi_J \right)\circ f_I  (x) \cdot  A \cdot \left( |b| +|\ell| \right) \cdot  \rho^N+ e^{a\cdot 2 \pi \cdot  c(\alpha_1 ^N,x)} \eta^{(\omega)}(I)  H \circ f_{I}  (x) \cdot 4\cdot \theta\cdot  \frac{|b|+|\ell|}{\epsilon'} \cdot 2\delta_3.$$
Since $H = \frac{ \left(\chi_J  H\right) }{\chi_J }\leq \frac{1}{1-\theta}\chi_J H$ we have
$$ \left|\nabla N_{s, \ell} ^J \left( H \right) (x) \right| \leq \left( \frac{\tilde{C}}{|b|+|\ell|}+\frac{\theta\cdot 8\delta_3}{(1-\theta)\epsilon'}+A \rho^{N} \right) (|b|+|\ell|) N_{s,\ell} ^J \left( H \right) (x) \leq A(|b|+|\ell|)N_{s,\ell} ^J \left( H \right) (x),$$
if $|b|+|\ell|$ is large enough, $|a|$ is sufficiently small, and 
$$\theta \leq \min \left( \frac{1}{2},\, \epsilon_1 \frac{A-1}{4\cdot 4\cdot \delta_3} \right) \text{ and } \rho^N \leq \frac{A-1}{2A}.$$

The previous estimate works for all $A>1$. Let us now suppose $f\in C^1(D)$ and $H\in \mathcal{C}_{A(|b|+|\ell|)}$ are such that
$$ |f|\leq H \text{ and } \left| \nabla f \right| \leq A(|b|+|\ell|)H.$$
We have
$$\left| \nabla P_{s, \ell, \omega, N} \left( f \right) (x) \right|  \leq \sum_{I\in X_N ^{(\omega)}} e^{a\cdot 2 \pi \cdot  c(I,x)} \left| (a+ib) 2 \pi\cdot \nabla \left( \log |f_I '| \right) (x) + \ell\cdot i \cdot \nabla \arg f_I '(x) \right|  \eta^{(\omega)}(I) H\circ f_I (x)$$
$$ + e^{a\cdot 2 \pi \cdot  c(I,x)} \eta^{(\omega)}(I) \left| \nabla \left( f\circ f_I \right) (x) \right|$$
$$ \leq \sum_{I\in \mathcal{A}^N} e^{a\cdot 2 \pi \cdot  c(I,x)}   \eta(I) H\circ f_I (x)\cdot \hat{C}\cdot (|b|+|\ell|) + e^{a\cdot 2 \pi \cdot  c(I,x)} \eta(I) H\circ f_I (x) \cdot A(|b|+|\ell|) \rho^N,$$
where the constant $\hat{C}$ does not depend on $N$, and $|b|+|\ell|$ is large enough (this is similar to Lemma \ref{Lemma 3.2 Oh} and \eqref{Eq tilde C}). Since $\chi_J ^{-1} \leq 2$ when  $\theta \leq \frac{1}{2}$, in this case $H\leq 2\chi_J  H$. Therefore, making this assumption
$$\left| \nabla P_{s, \ell, \omega, N} \left( f \right) (x) \right|  \leq A(|b|+|\ell|) N_{s,\ell} ^J \left( H \right) (x)$$
assuming $A \geq 4\hat{C}$ and $\rho^N \leq \frac{1}{4}$.

Finally, we can make our choices: $A\geq \max \left( 2, 4\hat{C} \right)$ , and  $N$ is large enough so  $\rho^N \leq \min \left( \frac{A-1}{2A}, \frac{1}{4} \right)$, and   $\theta \leq \min \left( \frac{1}{2}, \epsilon' \frac{A-1}{4\cdot 4\cdot \delta_3} \right)$.   Part  (3) follows from a similar argument with the same choice of parameters.

\end{proof}
Finally, we impose several additional conditions on $A, N, \epsilon_1,\theta$. Note that $\delta_1, \delta_2$, and so also $\delta_3$, are not free for us to choose.
\begin{enumerate}

\item  $8\cdot ||T||_{C^2} < A/2$, where $T$ is any local diffeomorphism as in Lemma \ref{Lemma 3.2 Oh}. 

\item $8\cdot   \rho^N \leq 1/2$. 

\item $200\cdot \epsilon_1\cdot  A\leq \log 2$.

\item $\frac{5\delta_1 \delta_2 \epsilon_1  }{2} - \frac{ 25 \epsilon_1 ^2  }{\delta_2}  - 40 A \epsilon_1 \rho^N   
\geq  2 \delta_1 \delta_2 \epsilon_1$. 

\item $2 \theta \leq \frac{(\delta_1 \delta_2 \epsilon_1)^2}{8\cdot 4^2}$.
\end{enumerate}

\subsubsection{Proof of Part 2 of Lemma \ref{Lemma 5.4 Naud}}
In this subsection we show that the cones constructed in the previous subsection are $L^2$-contracted by the operators $N_{s,\ell} ^J$. Our discussion is based on our one dimensional version \cite[Section 2.7.2]{algom2023polynomial} and the corresponding result of Oh-Winter \cite[Theorem 5.6 part (2)]{Oh2017winter}. First, we need   the following Definition. Recall the  definitions of $\tilde{\epsilon}, \delta_3,$ and  of $\chi_J$ from Section \ref{Section construction}.
\begin{Definition} \label{Def dense}
We call a subset $J\subseteq J_{s, \ell, \omega}$  dense if for all $1\leq i \leq q$  there exists $ j \in \lbrace 1, 2 \rbrace$ such that:
$$(j,1,i) \in J \text{ or } (j,2,i) \in J.$$
If $J$ is a dense subset we define
$$ W_J ' := \left \lbrace x_i:\, \exists j=1,2 \text{ with } (j,1,i)\in J \right \rbrace \bigcup \lbrace y_i:\, \exists j=1,2 \text{ with } (j,2,i)\in J \rbrace,$$
and
$$W_J := \bigcup_{z\in W_J '} B_{\tilde{\epsilon}\delta_3} (z).$$
\end{Definition}
The following Lemma, which is a version  of \cite[Lemma 5.7]{Naud2005exp} and \cite[Section 5.4]{Oh2017winter}, plays a key role in our analysis.
\begin{Lemma} \label{Lemma 5.7 Nuad} Suppose $J$ is a dense subset. Let $H\in C_{A \left( |b|+ |\ell| \right)}$. Then we can find some $\epsilon_2>0$ independent of $H,|b|+ |\ell|, \omega, N, J$  that satisfies
$$\int_{W_J} H \, d \mu_{\sigma^N \omega} \geq \epsilon_2 \cdot  \int_{K_{\sigma^N \omega}} H d  \mu_{\sigma^N \omega}.$$
\end{Lemma}
\begin{proof}
First, note that by Lemma \ref{Lemma Vitali} Part (1) and the choice of the $y_i$'s prior to \eqref{Eq 5.7 Oh},
$$K_{\sigma^N \omega} \subseteq \bigcup_{z\in W_J '} B_{100 \tilde{\epsilon}} (z).$$
Indeed, by Lemma \ref{Lemma Vitali} Part (1) for every $z\in K_{\sigma^N \omega}$ there exists $x_i$ such that $|x_i - z|< 50 \tilde{\epsilon}$. Since $|x_ i -y_i| < 5 \tilde{\epsilon}$, the result follows.

Also, since $H$ is $\log$ Lipschitz and by the choice of $\epsilon_1$, 
$$\sup \lbrace H(x):\, x\in B_{100\tilde{\epsilon} }(z) \rbrace \leq 2 \inf \lbrace H(x):\, x\in B_{100\tilde{\epsilon}}(z) \rbrace.$$
Let $C_1$ be the constant corresponding to the doubling constant $100/\delta_3$ as in as in Theorem \ref{Theorem disint} Part (7). Thus,
\begin{eqnarray*}
\int_{K_{\sigma^N \omega}} H d\mu_{\sigma^N \omega} &\leq &\sum_{z\in W_J '} \int_{B_{100 \tilde{\epsilon}} (z)} H d\mu_{\sigma^N \omega}\\
&\leq & \sum_{z\in W_J '} \sup \lbrace H(x):\, x\in B_{100\tilde{\epsilon}} \rbrace \cdot  \mu_{\sigma^N \omega} \left( B_{100 \tilde{\epsilon}} (z) \right)\\
&\leq & 2 \sum_{z\in W_J '} \inf  \lbrace H(x):\, x\in B_{100\tilde{\epsilon}} \rbrace \cdot  \mu_{\sigma^N \omega} \left( B_{100 \tilde{\epsilon}} (z) \right)\\
&\leq & 2C_1 \sum_{z\in W_J '} \inf  \lbrace H(x):\, x\in B_{100\tilde{\epsilon}} \rbrace \cdot  \mu_{\sigma^N \omega} \left( B_{\tilde{\epsilon}\delta_3} (z) \right)\\
&\leq & 2C_1 \int_{W_J} H d\mu_{\sigma^N \omega}. 
\end{eqnarray*}
Note the use  of the choice $C_1$ in the fourth inequality, and the definition of $W_J$ in the last one (Definition \ref{Def dense}). Since the constant $C_1$ is uniform (in particular, in $|b|+|\ell|$ and $\omega$), putting $\epsilon_2 = 2C_1$, the proof is complete.
\end{proof}

We can now define a subset of our parameter space alluded to in Lemma \ref{Lemma 5.4 Naud}:

$$\mathcal{E}_{s, \ell, \omega} := \lbrace J\subseteq J_{s, \ell, \omega}:\, J \text{ is a dense subset } \rbrace.$$

The following Lemma now proves Part (2) of Lemma \ref{Lemma 5.4 Naud}.
\begin{Proposition} \label{Prop 5.9 Naud} For some $0<\alpha<1$ that is uniform in $\omega$, for every $s=a+ib \in \mathbb{C},\ell\in \mathbb{Z},$ such that $|a|$ is small and $|b|+|\ell|$ is large, for every $H\in C_{A \left( |b|+|\ell|\right)}$ and every $J\in \mathcal{E}_{s, \ell, \omega}$,
$$\int_{K_{\sigma^N \omega}} \left| N_{s,\ell} ^J \left( H \right) \right|^2  d\mu_{\sigma^N \omega} \leq \alpha \cdot  \int_{K_{\sigma^N \omega}} P_{0,0,\omega,N} \left( H^2 \right) \,d \mu_{\sigma^N \omega}.$$

\end{Proposition}
This is a version of \cite[Proposition 2.21]{algom2023polynomial} and \cite[Theorem 5.6]{Oh2017winter}.
\begin{proof}
Fix $H\in C_{A\left( |b|+|\ell|\right)}$. For all $x\in D$, 
$$\left| N_{s,\ell} ^J \left( H \right) (x) \right|^2 \leq  \left( \sum_{I\in X_N ^{(\omega)}} e^{a\cdot 2 \pi \cdot  c(I,x)} \eta ^{(\omega)} (I) \chi_J ^2 \circ f_{I}(x) \right) \cdot  P_{a,0,\omega,N} \left( H^2 \right)(x).$$
Here we use the Cauchy-Schwartz and the definition of $N_{s,\ell} ^J$. Next, by definition for every $x\in W_J$ there is some $i\in \lbrace 1,2\rbrace$ so that
$$\chi_J \circ f_{\alpha_i ^N} (x) = 1-\theta.$$
Also, for all $I\in X_N ^{(\omega)}$, by \eqref{Eq tilde C} and \eqref{Eq C and C prime}
$$\left| a\cdot c(I,x) \right| \leq |a| D'  N.$$
So, with the notations of Section \ref{Section dis}, for all $x\in W_J$ 
$$  \sum_{I\in X_N ^{(\omega)}} e^{a\cdot 2 \pi \cdot  c(I,x)} \eta ^{(\omega)} (I) \chi_J ^2 \circ f_{I}(x) \leq e^{|a|\cdot D' \cdot N}-\theta\cdot e^{N\left( \min_{i\in I, 1\leq j \leq k_i} \log \mathbf{p}_j ^{(i)} - |a|\cdot D' \right)}.$$

It is always true that 
$$\int_{K_{\sigma^N \omega}} \left( N_{s,\ell} ^J \left( H \right) \right)^2  d \mu_{\sigma^N \omega} = \int_{W_J} \left( N_{s,\ell} ^J \left( H \right) \right)^2 d\mu_{\sigma^N \omega} + \int_{K_{\sigma^N \omega}\setminus W_J} \left( N_{s,\ell} ^J \left( H \right) \right)^2 d\mu_{\sigma^N \omega},$$
and that
$$  \sum_{I\in X_N ^{(\omega)}} e^{a\cdot 2 \pi \cdot  c(I,x)} \eta ^{(\omega)} (I) \chi_J ^2 \circ f_{I}(x) \leq e^{|a|\cdot D' \cdot N}.$$
Therefore, the discussion in the previous paragraph shows that
$$\int_{K_{\sigma^N \omega}} \left( N_{s,\ell} ^J \left( H \right) \right)^2 d\mu_{\sigma^N \omega} \leq  \left(e^{|a|\cdot D' \cdot N}-\theta\cdot e^{N\left( \min_{i\in I} \log \mathbf{p}_j ^{(i)} -|a|\cdot D' \right)} \right) \int_{W_J} P_{a,0,\omega,N} \left( H^2 \right) d \mu_{\sigma^N \omega}$$
$$ + e^{|a|\cdot D' \cdot N}\cdot \int_{K_{\sigma^N \omega}\setminus W_J} P_{a,0,\omega,N} \left( H^2 \right) d \mu_{\sigma^N \omega}$$
$$ = e^{|a|\cdot D' \cdot N}\cdot \int_{K_{\sigma^N \omega}} P_{a,0,\omega,N} \left( H^2 \right) d \mu_{\sigma^N \omega} - \theta\cdot e^{N\left( \min_{i\in I} \log \mathbf{p}_j ^{(i)} -|a|\cdot D' \right)}  \int_{W_J} P_{a,0,\omega,N} \left( H^2 \right) d \mu_{\sigma^N \omega}.$$

Since $H\in C_{A(|b|+|\ell|)}$ then $P_{a,0,\omega,N} \left( H^2 \right)\in C_{\frac{3}{4}A(|b|+|\ell|)}$ by Lemma \ref{Lemma cone}.   Thus, by Lemma \ref{Lemma 5.7 Nuad} applied to $P_{a,0,\omega,N} \left( H^2 \right)$ we have
$$\int_{K_{\sigma^N \omega}} \left( N_{s,\ell} ^J \left( H \right) \right)^2 d \mu_{\sigma^N \omega} \leq \left( e^{|a|\cdot D' \cdot N}- \epsilon_2\cdot \theta\cdot e^{N\left( \min_{i\in I} \log \mathbf{p}_j ^{(i)} -|a|\cdot D' \right)} \right) \int_{K_{\sigma^N \omega}} P_{a,0,\omega,N} \left( H^2 \right) d\mu_{\sigma^N \omega}.$$
Now, assuming $|a|$ is small enough we can find some $0<\alpha <1$ such that
$$\left( e^{|a|\cdot D' \cdot N}- \epsilon_2 \cdot \theta\cdot e^{N\left(  \min_{i\in I} \log \mathbf{p}_j ^{(i)} -|a|\cdot D' \right)} \right) \cdot e^{N\cdot D' \cdot |a|} \leq \alpha <1.$$
Indeed, here we use that
$$P_{a,0,\omega,N} \left( H^2 \right) \leq e^{N\cdot D' \cdot |a|} P_{0,0,\omega,N}\left( H^2 \right).$$
Finally,
$$\int_{K_{\sigma^N \omega}} \left( N_{s,\ell} ^J \left( H \right) \right)^2 d \mu_{\sigma^N \omega} \leq \alpha \cdot \int_{K_{\sigma^N \omega}} P_{0,0,\omega,N} \left( H^2 \right) d\mu_{\sigma^N \omega},$$
which proves our claim.

\end{proof}

\subsubsection{Proof of Part 3 of Lemma \ref{Lemma 5.4 Naud}}
In this subsection we prove the existence of dominating Dolgopyat operators, in the sense of Lemma \ref{Lemma 5.4 Naud} Part (3). This is the most involved part of the proof. At its heart  is the following Lemma, which is a higher dimensional version of \cite[Lemma 2.2]{algom2023polynomial}, and a model version of \cite[Lemma 5.10]{Naud2005exp}. It is however very different from its one dimensional counterparts. It is a model version of a corresponding result of Oh-Winter \cite[Lemma 5.8]{Oh2017winter}, on which it is based.
\begin{Lemma} \label{Lemma 5.10 Naud} Suppose $H\in C_{A\left(|b|+|\ell|\right)}$ and  $f\in C^1 \left( D \right)$ are functions that stratify
$$\left| f \right| \leq H, \text{ and } \left| \nabla f \right| \leq A \left( |b|+|\ell|\right)H.$$ 
For both $i=1,2$ we define functions $\Theta_j:D\rightarrow \mathbb{R}_+$ by
$$\Theta_1 (x):= \frac{ \left| e^{(a+ib)\cdot 2 \pi \cdot  c(\alpha_1 ^N,x)+i \cdot \ell \cdot \theta(\alpha_1 ^N,x)} \eta^{(\omega)}(\alpha_1 ^N) f \circ f_{\alpha_1 ^N} (x)+e^{(a+ib)\cdot 2 \pi \cdot  c(\alpha_2 ^N,x)+i \cdot \ell \cdot  \theta(\alpha_1 ^N,x)} \eta^{(\omega)}(\alpha_2 ^N) f\circ f_{\alpha_2 ^N} (x)     \right| }{ (1-2\theta)e^{a\cdot 2 \pi \cdot  c(\alpha_1 ^N,x)} \eta^{(\omega)}(\alpha_1 ^N) H \circ f_{\alpha_1 ^N} (x)+e^{a\cdot 2 \pi \cdot  c(\alpha_2 ^N,x)} \eta^{(\omega)}(\alpha_2 ^N) H\circ f_{\alpha_2 ^N} (x) },$$
and
$$\Theta_2 (x):= \frac{ \left| e^{(a+ib)\cdot 2 \pi \cdot  c(\alpha_1 ^N,x)+i \cdot \ell \cdot \theta(\alpha_1 ^N,x)} \eta^{(\omega)}(\alpha_1 ^N) f \circ f_{\alpha_1 ^N} (x)+e^{(a+ib)\cdot 2 \pi \cdot  c(\alpha_2 ^N,x)+i \cdot \ell \cdot  \theta(\alpha_1 ^N,x)} \eta^{(\omega)}(\alpha_2 ^N) f\circ f_{\alpha_2 ^N} (x)     \right| }{ e^{a\cdot 2 \pi \cdot  c(\alpha_1 ^N,x)} \eta^{(\omega)}(\alpha_1 ^N) H \circ f_{\alpha_1 ^N} (x)+(1-2\theta)e^{a\cdot 2 \pi \cdot  c(\alpha_2 ^N,x)} \eta^{(\omega)}(\alpha_2 ^N) H\circ f_{\alpha_2 ^N} (x) }.$$
Then for $\theta$ and $\epsilon_1$ small enough, and for all $|a|$ small enough and $|b|+|\ell|$ large enough, for every $1\leq j\leq q$ either $\Theta_1$ or $\Theta_2$ is $\leq 1$ on $B_{2\tilde{\epsilon}\delta_3} (x_j)$ or $B_{2\tilde{\epsilon}\delta_3} (y_j)$.
\end{Lemma}
We first recall a few Lemmas from \cite{Oh2017winter}. Recall that $\epsilon_1$ was defined in Section \ref{Section construction}.
\begin{Lemma} \cite[Lemma 5.3]{Oh2017winter} \label{Lemma 5.3 Oh}
 Fix functions $H,f$ as in Lemma \ref{Lemma 5.10 Naud}. Then for every $x\in D$ and $f_i\in \Phi^N$, either $\left| f\circ f_i \right| \leq \frac{3}{4}H\circ f_i$ on $B_{10\epsilon_1/A} (x)$, or $\left| f\circ f_i \right| \geq \frac{1}{4}H\circ f_i$  on $B_{10\epsilon_1/A} (x)$.
\end{Lemma}
Recall that for $0\neq z\in \mathbb{C}$, $\arg (z)\in (-\pi,\, \pi]$ is the unique real number such that $|z|e^{i\arg(z)}=z$.
\begin{Lemma} \cite[Lemma 5.2]{Oh2017winter} \label{Lemma 5.2 Oh} Fix $0\neq z_1,z_2 \in \mathbb{C}$ that satisfy
$$\frac{|z_1|}{|z_2|} \leq 1 \text{ and } \left| \arg \left( \frac{z_1}{z_2}\right) \right| \geq \eta \in [0,\,\pi].$$
Then 
$$\left| z_1 + z_2 \right| \leq (1-\frac{\eta^2}{8})|z_1| + |z_2|.$$

\end{Lemma}
$$ $$
\noindent{ \textbf{Proof of Lemma \ref{Lemma 5.10 Naud}}} 
Fix $1\leq i \leq q$. If  either $f_{\alpha_1 ^N}$ or $f_{\alpha_2 ^N}$ satisfy the first alternative of Lemma \ref{Lemma 5.3 Oh}  on $B_{10\tilde{\epsilon}} (x_i)$ then we are done. So, we assume the converse.

For $x\in B_{10\tilde{\epsilon}} (x_i)$ we denote by $\text{Arg} (x)$ the argument, taking values in $[-\pi,\pi)$, of the ratio of the summands in the numerator of $\Theta_j$:
\begin{equation} \label{Eq (5.23)}
\text{Arg} (x) := \arg \left( \frac{  e^{ib\cdot 2 \pi \cdot  c(\alpha_1 ^N,x)+i \cdot \ell \cdot \theta(\alpha_1 ^N,x)} \eta^{(\omega)}(\alpha_1 ^N) f \circ f_{\alpha_1 ^N} (x)    }{e^{ib\cdot 2 \pi \cdot  c(\alpha_2 ^N,x)+i\cdot \ell \cdot \theta(\alpha_2 ^N,x)} \eta^{(\omega)}(\alpha_2 ^N) f\circ f_{\alpha_2 ^N} (x) } \right).
\end{equation}
We show that $\text{Arg} (x)$ cannot be small near both $x_i$ and $y_i$; this will imply the desired cancellations in the numerator and give the Lemma. The key is the evaluation of the corresponding derivatives from below, that we now carry out.

\begin{Remark}
It will not be hard to see from our argument that we may assume $|\text{Arg} (x)|\leq \frac{\pi}{2}$ on  $B_{10\tilde{\epsilon}} (x_i)$. This is similar to \cite[the Remark after eq. (5.23)]{Oh2017winter}. In particular, the $\arg$ function can be assumed to be correspondingly continuous. 
\end{Remark}

Note that
\begin{equation} \label{Eq 5.24 Oh}
 \text{Arg} (x) = b\left(  c(\alpha_1 ^N,x) -  c(\alpha_2 ^N,x) \right) +\ell\left( \theta(\alpha_1 ^N,x) - \theta(\alpha_2 ^N,x) \right)+\arg \left( \frac{f\circ f_{\alpha_1 ^N} (x)}{f\circ f_{\alpha_2 ^N} (x)} \right)
\end{equation}
\begin{equation} \label{Eq 5.25 Oh}
= b\left(  c(\alpha_1 ^N,x) -  c(\alpha_2 ^N,x) \right) +\ell\left( \theta(\alpha_1 ^N,x) - \theta(\alpha_2 ^N,x) \right)+\arg \left( f\circ f_{\alpha_1 ^N} (x) \right) - \arg \left( f\circ  f_{\alpha_2 ^N} (x) \right),
\end{equation}
where we think of these numbers as elements in $S^1 \simeq \mathbb{T}$ when necessary. Note that
\begin{equation}
\left| \nabla \arg \left( f\circ f_{\alpha_j ^N} \right) \right| \leq 4A\left(|b|+|\ell|\right)\cdot \rho^N,\quad \text{ for } j=1,2.
\end{equation} 
Applying Taylor's expansion at the point $x_i$ to the map 
$$T(y)= b\left(  c(\alpha_1 ^N,y) -  c(\alpha_2 ^N,y) \right) +\ell\left( \theta(\alpha_1 ^N,y) - \theta(\alpha_2 ^N,y) \right),$$ 
using that $\vert \vert T \vert \vert_{C^2} \leq \frac{1}{2\delta_2}$ by Lemma \ref{Lemma 3.2 Oh}, we obtain
$$\left|  T(y_i) - T(x_i) \right| \geq \left| \left\langle x_i-y_i,\,\hat{w_i} \right\rangle \right| \cdot \left| w_i \right| - \frac{ \left( |b|+|\ell| \right) \left| x_i - y_i \right|^2 }{\delta_2}. $$
Therefore, by \eqref{Eq 5.7 Oh}, \eqref{Eq 5.7 oh 2}, and recalling that $y_i \in B_{5 \tilde{\epsilon}} (x_i)$, and by the assumptions imposed after \eqref{Eq 5.7 oh 2} and Lemma \ref{Lemma cone},
\begin{eqnarray*}
\left| \text{Arg} (y_i) - \text{Arg} (x_i) \right| &\geq& \left|  T(y_i) -  T(x_i) \right|  - \left| x_i - y_i\right| \left( 8A\left(|b|+|\ell|\right)\rho^N \right) \\
&\geq& \left| \left\langle x_i-y_i,\,\hat{w_i} \right\rangle \right| \cdot \left| \hat{w_i} \right| - \frac{ \left( |b|+|\ell| \right) \left| x_i - y_i \right|^2 }{\delta_2}  - \left| x_i - y_i\right| \left( 8A\left(|b|+|\ell|\right)\rho^N \right) \\
&\geq& \frac{\delta_1 \cdot L \cdot B_1 \delta_2 \tilde{\epsilon} \left( |b|+|\ell|\right) }{2} - \frac{ \left( |b|+|\ell| \right) \left| x_i - y_i \right|^2 }{\delta_2}  - \left| x_i - y_i\right| \left( 8A\left(|b|+|\ell|\right)\rho^N \right) \\
&\geq& \frac{5\delta_1 \delta_2 \epsilon_1  }{2} - \frac{ 25 \epsilon_1 ^2  }{\delta_2}  - 40 A \epsilon_1 \rho^N   \\
&\geq & 2 \delta_1 \delta_2 \epsilon_1.
\end{eqnarray*}
Note specifically the use of \eqref{Eq 5.7 Oh} in the third inequality.

The next step is to argue that
$$\left| \text{Arg} (y) - \text{Arg} (x) \right| > \delta_1 \delta_2 \epsilon_1,\quad \text{ for all } x\in B_{2\delta_3 \tilde{\epsilon}} (x_i),\,  y\in B_{2\delta_3 \tilde{\epsilon}} (y_i).$$
First, we estimate the $C^1$ norm of $\text{Arg} (x)$. By the additional assumptions put after Lemma \ref{Lemma cone}, 
\begin{equation} \label{Eq 5.27 Oh}
\left| \nabla \text{Arg} (x)\right| \leq 8\cdot ||T||_{C^1}\cdot \left( |b|+|\ell|\right)+ 8A\left(|b|+|\ell|\right)\rho^N  \leq A\left( |b|+|\ell|\right).
\end{equation}
So, by the choice of $\delta_3$ as in Section \ref{Section construction}
$$\left| \text{Arg} (x_i) - \text{Arg} (x) \right|\leq 2A \delta_3 \tilde{\epsilon} \left( |b|+|\ell|\right)\leq \frac{\delta_1 \delta_2 \epsilon_1}{2}.$$
Similarly,
$$\left| \text{Arg} (y_i) - \text{Arg} (y) \right|\leq \frac{\delta_1 \delta_2 \epsilon_1}{2}.$$
It follows that for all $x\in B_{2\delta_3 \tilde{\epsilon}} (x_i),\,  y\in B_{2\delta_3 \tilde{\epsilon}} (y_i)$ we have
$$\left| \text{Arg} (y) - \text{Arg} (x) \right| \geq \left| \text{Arg} (y_i) - \text{Arg} (x_i) \right| -\left| \text{Arg} (y) - \text{Arg} (y_i) \right|-\left| \text{Arg} (x) - \text{Arg} (x_i) \right| \geq \delta_1 \delta_2 \epsilon_1.$$

We now claim one of the following alternatives must hold:
\begin{enumerate}
\item $\text{Arg} (x) > \frac{\delta_1 \delta_2 \epsilon_1}{4}$ on $B_{2\delta_3 \tilde{\epsilon}} (x_i)$; or
\item $\text{Arg} (y) > \frac{\delta_1 \delta_2 \epsilon_1}{4}$ on $B_{2\delta_3 \tilde{\epsilon}} (y_i)$.
\end{enumerate}
Indeed,  if (1) fails then there exists $x\in B_{2\delta_3 \tilde{\epsilon}} (x_i)$ such that $\text{Arg} (x) \leq \frac{\delta_1 \delta_2 \epsilon_1}{4}$. Therefore, for all $y\in B_{2\delta_3 \tilde{\epsilon}} (y_i)$,
$$\left| \text{Arg} (y) \right|\geq \left| \text{Arg} (x)- \text{Arg} (y)  \right| - \left| \text{Arg} (x) \right| \geq \frac{\delta_1 \delta_2 \epsilon_1}{2}.$$

This claim now implies Lemma \ref{Lemma 5.10 Naud} via Lemma \ref{Lemma 5.2 Oh}: Suppose without the loss of generality that (2) holds, and that $x\in B_{2\delta_3 \tilde{\epsilon}} (y_i)$ has
$$\left| e^{a \cdot 2 \pi \cdot  c(\alpha_2 ^N,x)} \eta^{(\omega)}(\alpha_2 ^N) f \circ f_{\alpha_2 ^N} (x) \right| \leq \left| e^{a \cdot 2 \pi \cdot  c(\alpha_1 ^N,x)} \eta^{(\omega)}(\alpha_1 ^N) f \circ f_{\alpha_1 ^N} (x) \right|.$$
Then by Lemma \ref{Lemma 5.2 Oh}, 
$$ \left| e^{(a+ib)\cdot 2 \pi \cdot  c(\alpha_1 ^N,x)+i\cdot \ell \cdot \theta(\alpha_1 ^N,x)} \eta^{(\omega)}(\alpha_1 ^N) f \circ f_{\alpha_1 ^N} (x)+e^{(a+ib)\cdot 2 \pi \cdot  c(\alpha_2 ^N,x)+i\cdot \ell \cdot \theta(\alpha_1 ^N,x)} \eta^{(\omega)}(\alpha_2 ^N) f\circ f_{\alpha_2 ^N} (x)     \right| $$
$$\leq e^{a\cdot 2 \pi \cdot  c(\alpha_1 ^N,x)} \eta^{(\omega)}(\alpha_1 ^N) \left| |f \circ f_{\alpha_1 ^N} (x) \right| +\left(1-\frac{(\delta_1 \delta_2 \epsilon_1)^2}{8\cdot 4^2} \right)\cdot  e^{a\cdot 2 \pi \cdot  c(\alpha_2 ^N,x)} \eta^{(\omega)}(\alpha_2 ^N) \left| f\circ f_{\alpha_2 ^N} (x) \right|$$
$$\leq  e^{a\cdot 2 \pi \cdot  c(\alpha_1 ^N,x)} \eta^{(\omega)}(\alpha_1 ^N) H \circ f_{\alpha_1 ^N} (x)+\left(1-\frac{(\delta_1 \delta_2 \epsilon_1)^2}{8\cdot 4^2} \right)\cdot e^{a\cdot 2 \pi \cdot  c(\alpha_2 ^N,x)} \eta^{(\omega)}(\alpha_2 ^N) H\circ f_{\alpha_2 ^N} (x).$$
This implies the Lemma by the choice of $\theta$ (after Lemma \ref{Lemma cone}). The other cases are  analogues. The proof is complete.

 $\hfill{\Box}$
$$ $$

\noindent{ \textbf{Proof of Lemma \ref{Lemma 5.4 Naud} Part (3)} Let $N,A,\epsilon_1,\theta$ be such that they meet the conditions required so that parts (1) and (2) of Lemma \ref{Lemma 5.4 Naud}, and  Lemma \ref{Lemma 5.10 Naud}, all hold true. Fix $f\in C^1(D),H\in C_{A\left( |b|+|\ell|\right)}$ that satisfies
$$\left| f \right| \leq H,\text{ and } \left| \nabla f \right| \leq A\left( |b|+|\ell|\right)H.$$
We construct $J\in \mathcal{E}_{s,\ell, \omega}$,  a dense subset, so that we have
$$|P_{s, \ell, \omega, N} f| \leq N_{s,\ell} ^J H \quad \text{ and } |\nabla (P_{s, \ell, \omega, N} f)|\leq A\left( |b|+|\ell|\right)N_{s, \ell} ^J H.$$
By Lemma \ref{Lemma cone} the latter is true for every $J\neq \emptyset$. Let us now establish the former.

For each $1 \leq j \leq q$ we proceed as follows:
\begin{enumerate}
\item If $\Theta_1\leq 1$ on $B_{2\delta_3 \tilde{ \epsilon}} (x_i)$ include $(1,1,j)\in J$, else
\item If $\Theta_2\leq 1$ on $B_{2\delta_3 \tilde{ \epsilon}} (x_i)$ include $(2,1,j)\in J$, else
\item If $\Theta_1\leq 1$ on $B_{2\delta_3 \tilde{ \epsilon}} (y_i)$ include $(1,2,j)\in J$, else
\item If $\Theta_2\leq 1$ on $B_{2\delta_3 \tilde{ \epsilon}} (y_i)$ include $(2,2,j)\in J$.
\end{enumerate}
By Lemma \ref{Lemma 5.10 Naud}} at least one of these occurs for each $j$. Thus, $J$ is dense. Furthermore,
$$|P_{s, \ell, \omega, N} f| \leq N_{s,\ell} ^J H $$
is an immediate consequence of the corresponding inequality of Lemma \ref{Lemma 5.10 Naud}, and of the definition of $N_{s,\ell} ^J$ as in Section \ref{Section construction}.
 The proof is complete. $\hfill{\Box}$

\section{An exponentially fast equidistribution phenomenon for the renormalized derivative cocycle} \label{Section equi}
As we discussed in the Introduction, estimates like the one in Theorem \ref{Theorem spectral gap} have wide ranging applications. In this paper will use it to derive an equidistribution result for the derivative cocycle (made up from \eqref{The der cocycle} and \eqref{The angle cocycle}). Morally, it means that, if done correctly, a coupling of the re-centred norm cocycle with the angle cocycle equidistributes towards an absolutely continuous on $J\times \mathbb{T}$ where $J\subseteq \mathbb{R}$ is some interval. This result will be instrumental when proving the Fourier decay bound from Theorem \ref{Main Theorem analytic}. We remark that role of the model construction from Theorem \ref{Theorem disint} is done, and from now on we work with our IFS directly. 

Let us now be more specific. We retain the notations and assumptions as in Section \ref{Section spectral gap}. In particular, our IFS satisfies the conditions of Theorem \ref{Theorem disint}; therefore, Theorem \ref{Theorem spectral gap}  holds  true. Let us further recall $\Phi$ is $C^\omega(\mathbb{C})$, and  $\nu_\mathbf{p}$ is our self-conformal measure where $\mathbf{p}$ is a strictly positive probability vector. Recall that we are working with the Bernoulli measure on $\mathcal{A}^\mathbb{N}$ given by $\mathbb{P}=\mathbf{p}^\mathbb{N}$.

Let us first define symbolic versions of the derivative cocycle(s) defined in Section \ref{Section induced IFS}. With $G$ being the semigroup as in that Section,
 we define on $G\times \mathcal{A}^\mathbb{N}$  symbolic analogues via
\begin{equation} \label{The symbolic der cocycle}
\tilde c(I,\omega)=-\log | f'_I(x_\omega)|\in \mathbb{R}_+,\quad \tilde \theta(I,\omega)=\arg \left( f'_I (x_\omega) \right) \in \mathbb{T}.
\end{equation}

Here is the   random walk we consider. For all $n\in \mathbb{N}$ we define  functions  on $\mathcal{A}^\mathbb{N}$ by
\begin{equation} \label{Eq for Sn}
S_n(\omega)=-\log \left| f'_{\omega|_n}(x_{\sigma^n(\omega)})\right|,\quad A_n(\omega) = \arg \left( f'_{\omega|_n}(x_{\sigma^n(\omega)}) \right).
\end{equation}
Let $X_1, Y_1$ be the random variables that describe the first step
\begin{equation} \label{Eq for X1}
X_1(\omega):= \tilde c(\omega_1,\sigma(\omega))=-\log f'_{\omega_1}(x_{\sigma(\omega)}), \quad Y_1 (\omega):=\tilde \theta(\omega_1,\sigma(\omega))=\arg f'_{\omega_1}(x_{\sigma(\omega)}).
\end{equation}
For all $n\in\mathbb{N}$ we can define the $n$th-step random variables by 
$$X_n(\omega) :=  - \log f_{\omega_n} ' \left( x_{\sigma^n (x_\omega)} \right)  = X_1 \circ \sigma^{n-1},\quad Y_n(\omega):=Y_1\circ \sigma^{n-1}$$
Let $\kappa$ be the distribution of  $X_1$, and let $\beta$ be the distribution of $Y_1$. Then  $X_n \sim \kappa$ and $Y_n\sim \beta$ for all $n$.  By \eqref{Eq C and C prime}  there are positive constants $D,D'\in \mathbb{R}$ such that   
$$\kappa \in \mathcal{P}([D,D']).$$
Note that the support of $\kappa$ is bounded away from $0$. Therefore, for all integer $n$ and $\omega\in A^\mathbb{N}$,
$$S_n(\omega) = \sum_{i=1} ^n X_i (\omega),\quad A_n(\omega) = \sum_{i=1} ^n Y_i (\omega) \mod 1.$$
So, both $S_n$ and $A_n$ can be seen as random walks. 

We proceed to define a normalization (stopping-time like) function  on $\mathcal{A}^\mathbb{N}$. Let $k>0$ and define
$$ \tau_k(\omega):=\min\{n: S_n(\omega) \geq k\}.$$
The input $k$  can take non-integer values.  Then for all $k>0$ and every $\omega\in A^\mathbb{N}$ we have, by  \eqref{Eq C and C prime}, 
$$-\log |f_{\omega|_{\tau_k(\omega)}}'(x_{\sigma^{\tau_k(\omega)}(\omega)})|= S_{\tau_k(\omega)} (\omega) \in [k, k+D'].$$
We also define a local $C^8$ norm on $\mathbb{T}\times[-1-D, D'+1]$ by, for $g\in C^8(\mathbb{T}\times[-1-D, D'+1])$,
\begin{equation}\label{Def local norm}
||g||_{C^8} = \max_{0\leq i,j \leq 4}  \left| \frac{\partial g}{\partial x^i \partial y^j} \right|_\infty.
\end{equation}
The following result provides  effective equidistribution  for a coupling of the random variables $S_{\tau_k}-k, A_{\tau_k}$. It will be instrumental in proving Theorem \ref{Main Theorem analytic} in Section \ref{Section Fourier decay}.
\begin{theorem} \label{Theorem equi} Let $\epsilon>0$ be the constant given by Theorem \ref{Theorem spectral gap}.   For all $k>D'+1$, \newline $g\in C^8(\mathbb{T}\times[-1-D, D'+1])$, and $\omega \in \mathcal{A}^\mathbb{N}$, we have:
$$\mathbb{E}\left( g \left( A_{\tau_k(\eta)} (\eta),\, S_{\tau_k(\eta)} (\eta)-k \right) \big|\, \sigma^{\tau_k(\eta)} \eta  =\omega \right) $$
$$= \frac{1}{\chi} \int_{\mathcal{A}^\mathbb{N}} \int_\mathcal{A} \int_\mathbb{T} \int_{-\tilde{c}(i,\omega)} ^0 g(s+\tilde{\theta}(i,\omega),x+\tilde{c}(i,\omega))\,dx\, d \lambda_\mathbb{T}(s)\, d\mathbf{p}(i)\, d\mathbb{P}(\omega)+e^{-\frac{\epsilon k}{2}} O\left(  ||g||_{C^8} \right).$$
\end{theorem}
Theorem \ref{Theorem equi} is  a higher dimensional  version of \cite[Theorem 4.1]{algom2023polynomial}. To be more precise, in that Theorem we proved a version of Theorem \ref{Theorem equi} only for the norm cocycle. What is new and improved here is that this holds when we couple this cocycle with the angle cocycle. We also note that this is a conformal and smooth  version of a result of Li-Sahlsten \cite[Proposition 3.7]{Li2020Sahl} for self-affine IFSs.

For the proof, the main ingredient is a renewal Theorem with exponential speed, that we prove in the next subsection. In fact, once this result is established, one can derive Theorem \ref{Theorem equi} by what are nowadays standard arguments. Such arguments are originally due to Li \cite{li2018fourier}, and have subsequently appeared in \cite{algom2023polynomial, Li2020Sahl, li2018fourier, li2019trigonometric}, to name just a few examples. So, after our renewal Theorem is done (Section \ref{Section renewal}), we will discuss the only part that is different here compared with \cite[proof of Theorem 4.1]{algom2023polynomial}, which is the residue process. This discussion takes place in Section \ref{Section residue process}. The rest of the proof is similar to \cite[Theorem 4.1]{algom2023polynomial} and is thus omitted\footnote{In fact, the proof of this result appears with full details in the original (V1) version of the arXiv submission of this paper.}.

\subsection{A renewal theorem with exponential speed} \label{Section renewal}
We begin by stating the following Theorem, whose new parts are almost formal consequences of Theorem \ref{Theorem spectral gap}. We remind the reader that $\chi$ is the Lyapunov exponent of $\omega \mapsto c(\omega_1,\sigma(\omega))$.
\begin{theorem} \label{Theorem U}
Fix the constants $\epsilon$ and $\gamma>0$ given by  Theorem \ref{Theorem spectral gap}. Then:
\begin{enumerate}
\item \cite[Lemma 11.17]{Benoist2016Quint} For all $|a|<\epsilon$ and $b\in \mathbb{R},\ell \in \mathbb{Z}$, the transfer operator $P_{a+ib, \ell}$ is a bounded operator on $C^1(D)$. Furthermore,  it is analytic in $(a+ib,\, \ell)$.

\item \cite[Proposition 4.1]{Li2018decay} For every $\ell \in \mathbb{Z}$  there is an analytic operator $U_\ell(a+ib)$ on $C^1(D)$ where $a+ib \in \mathbb{C}, |a|<\epsilon$ and  $\ell\in \mathbb{Z}$  such that,  if $\ell=0$
$$(I-P_{-(a+ib),0})^{-1} = \frac{1}{\chi(a+ib)}N_0+U_0(a+ib),\, \text{ where we define } N_0(f):=\int f \, d \nu,$$
and otherwise
$$(I-P_{-(a+ib),\ell})^{-1} = U_\ell(a+ib).$$

\item \cite[Proposition 4.26]{li2018fourier} There exists some $C>0$ such that if $|a|<\epsilon$ then for all $a+ib \in \mathbb{C}$ and every $\ell \in \mathbb{Z}$,
$$||U_\ell(a+ib)||\leq C(1+\left| b\right| + \left| \ell \right| )^{1+\gamma}, \, \text{ w.r.t. the operator norm from \eqref{Eq B-Q norm}.} $$
\end{enumerate}
\end{theorem}
This Theorem is a higher dimensional version of \cite[Theorem 3.1]{algom2023polynomial}. The key (and only) difference lies with existence of the integer frequency $\ell$. Other than that the idea and proof is similar, and follows along the proofs as in e.g. \cite[Chapter 11]{Benoist2016Quint}, specifically \cite[Lemma 11.18]{Benoist2016Quint}. Part (3) is of particular importance to us, and this is where Theorem \ref{Theorem spectral gap} does its magic,  by the  identity
$$(I-P_{a+ib, \ell})^{-1} = \sum_{n=0} ^\infty P_{a+ib, \ell} ^n.$$
See e.g. \cite[Proposition 4.26]{li2018fourier} or \cite[Theorem 3.1]{algom2023polynomial} for more details.

Let us now discuss renewal theory. Let $f$ be a  non-negative bounded function on $D\times \mathbb{T} \times \mathbb{R}$ and  let $(z,t)\in D\times \mathbb{R}$. We define, recalling \eqref{The angle cocycle} and \eqref{The der cocycle}, a renewal operator by
$$Rf(z,\, t):= \sum_{n=0} ^\infty \int f(\eta.z,\, \theta( \eta,\,z),\,  c(\eta,\, z) -t) d\mathbf{p}^n (\eta).$$
This is well defined as $f$ is non-negative. For all $(z,x,y)\in D\times \mathbb{T}\times \mathbb{R}$ let $f_z :\mathbb{T}\times \mathbb{R} \rightarrow \mathbb{R}$  be the function
$$f_z (x,y):=f(z,x,y).$$
Next, for every $(z, \ell, v) \in D\times \mathbb{Z}\times \mathbb{C}$ we define the Fourier transform 
$$\hat{f}(z, \ell, v) := \int e^{i v y}\cdot e^{i \ell u} f(z,u,y)\,du\,dy=\hat{f_z}(\ell, v).$$
The next Proposition is the key ingredient towards Theorem \ref{Theorem equi}. It is a higher dimensional version of \cite[Proposition 3.2]{algom2023polynomial}. However, note that there are some substantial differences in both the statement and the proof due to our higher dimensional setting (in particular, the dependence on the $\ell$).
\begin{Proposition} \label{Proposition renewal}
Fix the constant $\epsilon$ given by Theorem \ref{Theorem spectral gap}, and suppose $C\subseteq \mathbb{T}\times \mathbb{R}$ is compact. Let $f$ be a function on $D\times \mathbb{T} \times \mathbb{R}$ that is non-negative, bounded, and continuous . Suppose $f_z \in C_C ^{8} (\mathbb{T}\times \mathbb{R})$ for all $z\in D$, and that $\hat{f}_z (\ell, v) \in C^1(D)$ for all  $(\ell, v)\in \mathbb{Z}\times \mathbb{C}$. Assume
$$\sup_{z\in D}  \left( \left|\frac{\partial}{\partial x^{4} \partial y^{4}} f_z \right|_{L^1} + \left|f_z\right|_{L^1} \right) <\infty.$$
Then  for all $z\in D$ and $t\in \mathbb{R}$
$$Rf(z, t) = \frac{1}{\chi} \int_D \int_{\mathbb{T}} \int_{-t} ^\infty f(z,x,y)\,dy\, dx\, d\nu(z) + e^{-\epsilon\cdot |t|}O\left(e^{\epsilon | \supp f_{z}|} \left(  \left|\frac{\partial}{\partial x^{4} \partial y^{4}} f_z \right|_{L^1} + \left|f_z\right|_{L^1}  \right) \right)$$
where $|\text{supp}(f_{ })|$ is the sup over  the norms of the elements of $\supp(f_{  z} )$.
\end{Proposition}
We remark that Proposition \ref{Proposition renewal} is a higher dimensional IFS-type version of the renewal Theorem of Li \cite[Theorem 1.1 and Proposition 4.27]{li2018fourier}. It can also be seen as a conformal analogue (i.e. without a proximality assumption) of the renewal Theorem  of  Li-Sahlsten \cite[Proposition 3.4]{Li2020Sahl} for self-affine IFSs.

\begin{proof}
First, we claim that for all $(z,t)\in D\times \mathbb{R}$,
\begin{equation} \label{Eq Lemma 4.6 Li}
Rf(z, t) = \frac{1}{\chi}  \int_{D} \int_\mathbb{T}  \int_{-t} ^\infty  f(z,x,y)\,dy\, dx\, d\nu(z) + \lim_{a\rightarrow 0^+} \frac{1}{2\pi} \int_\mathbb{Z} \int_\mathbb{R} e^{-it b}  U_\ell(a-ib)\hat{f}( z,\ell, b)\,db \, d\ell.
\end{equation}
This is similar to  \cite[Proposition 4.14]{boyer2016rate} and \cite[Lemma 4.6]{Li2018decay}.

Let $a\geq 0$ and define
$$B_a f(z,t):= \int e^{-a \sigma(i,\, z)} f(\eta.z,\, \theta(i,z),\, c(i,z)+t)d\mathbf{p}(i).$$
Let $B:=B_0$. Then 
$$Rf(z,-t) = \sum_{n\geq 0} B^n f(z,t).$$
As $f\geq 0$  by  monotone convergence 
$$\lim_{a\rightarrow 0^+} \sum_{n\geq 0 } \int e^{-a \sigma(\eta,\, z)} f(\eta.z,\, \theta( \eta,\, z),\, \sigma(\eta,\,z)+t) d\mathbf{p}^n (\eta) = \sum_{n\geq 0 } \int  f(\eta.z,\, \theta( \eta,\, z),\, \sigma(\eta,\,z)+t) d\mathbf{p}^n (\eta).$$
Therefore
\begin{equation} \label{Eq (4.3) Li}
\sum_{n\geq 0} B^n (f)(z,t) = \lim_{a\rightarrow 0^+} \sum_{n\geq 0 } B_a ^n (f) (z,t).
\end{equation}

For every $d\in D$ we are assuming that $f_z \in C_C ^{8} (\mathbb{T}\times \mathbb{R}) \subset L^1 (\mathbb{T}\times \mathbb{R})$. By the inverse Fourier transform, we thus have for every $a>0$
$$\sum_{n\geq 0 } B_a ^n (f) (z,t) = \sum_{n\geq 0 } \int e^{-a \sigma(\eta,\, z)} f(\eta.z,\, \theta( \eta,\, z),\, \sigma(\eta,\,z)+t) d\mathbf{p}^n (\eta)$$
\begin{equation} \label{Eq (4.4) Li}
= \sum_{n\geq 0 } \int e^{-a \sigma(\eta,\, z)}  \frac{1}{2\pi} \int_\mathbb{Z} \int_\mathbb{R} e^{ ib( \sigma(\eta,z)+t)}e^{i\ell \cdot \theta( \eta,\, z)} \hat{f}(\eta.z,\,\ell,\, b)\, db\, d\ell\, d\mathbf{p}^n (\eta).
\end{equation}
We claim that \eqref{Eq (4.4) Li} is absolutely convergent: As $f_z$ is compactly supported for every $z\in D$, for every $(z,\ell)\in D\times \mathbb{Z}$, $\hat{f_z}(\ell,\,\cdot)\in C^\omega (\mathbb{C})$. For all $z\in D$, $k_1,k_2\in \mathbb{N}\cup \lbrace 0 \rbrace$, $\delta\geq 0$, $\ell \in \mathbb{Z}\setminus \lbrace 0 \rbrace$ and $b\in \mathbb{R}$,
$$\left| \hat{f_z}(\ell, \pm i\delta+b) \right| \leq e^{\delta | \supp f_z|} \frac{1}{|\theta|^{k_1}\cdot |\ell|^{k_2}} \left|\frac{\partial}{\partial x^{k_2} \partial y^{k_1}} f_y \right|_{L^1}.$$
For $\ell=0$ a similar corresponding bound holds. Therefore, inputting $k_1=k_2=0$ and $k_1=k_2=4$, 
\begin{equation} \label{Eq 4.60}
\left|  \hat{f_z}(\pm i\delta+b) \right|  \leq e^{\delta | \supp f_z|} \frac{2}{|b|^{4}\cdot |\ell|^4+1} \left( \left|\frac{\partial}{\partial x^{4} \partial y^{4}} f_z \right|_{L^1} + |f_z|_{L^1} \right)
\end{equation}

So, for $a>0$, defining 
$$C_0=\int_\mathbb{R} \frac{2}{|b|^{4}+1}\, db$$
and 
$$C_\ell=\int_\mathbb{R} \frac{2}{|b|^{4}\cdot |\ell|^4+1}\, db$$
otherwise, we have:
$$\sum_{n\geq 0 } \int e^{-a \sigma(\eta,\, z)}  \frac{1}{2\pi} \int_\mathbb{Z} \int_\mathbb{R} e^{ ib( \sigma(\eta,z)+t)}e^{i\ell \cdot \theta( \eta,\, z)} \hat{f}(\eta.z,\,\ell,\, b)\, db\, d\ell\, d\mathbf{p}^n (\eta) $$
$$\leq  C_0\cdot e^{\delta | \supp f_z|} \cdot  \sup_{y}  \left( |f_z ^{(4)}|_{L^1} + |f_y |_{L^1} \right) \cdot  \sum_{n\geq 0 } \int e^{-a \sigma(\eta,\, z)}\cdot d\mathbf{p}^n (\eta)$$
$$+\sum_{\ell\neq 0} C_\ell\cdot e^{\delta | \supp f_z|} \cdot  \sup_{z}  \left( \left|\frac{\partial}{\partial x^{4} \partial y^{4}} f_z \right|_{L^1} + |f_z |_{L^1} \right) \cdot  \sum_{n\geq 0 } \int e^{-a \sigma(\eta,\, z)}\cdot d\mathbf{p}^n (\eta)$$
$$\leq  C\cdot e^{\delta | \supp f_z|} \cdot  \sup_{z}  \left( |f_z ^{(4)}|_{L^1} + |f_z|_{L^1} \right) \cdot  \sum_{n\geq 0 }  e^{-a D\cdot n} $$
$$+ \sum_{\ell \neq 0} C_\ell\cdot e^{\delta | \supp f_y|} \cdot  \sup_{z}  \left(  \left|\frac{\partial}{\partial x^{4} \partial y^{4}} f_z \right|_{L^1} + |f_z|_{L^1} \right) \cdot  \sum_{n\geq 0 }  e^{-a D\cdot n}<\infty.$$
Note that we used the fact that the $C_\ell$s are summable.

The absolute convergence in  \eqref{Eq (4.4) Li} is thus proved. Thus,  Fubini can be applied  to  change the order of integration. As for all $(\ell, v) \in \mathbb{T}\times \mathbb{C}$ we have $\hat{f}(\cdot ,\ell,v) \in C^1 (D)$, Theorem \ref{Theorem U} applies and we have
\begin{eqnarray*}
\sum_{n\geq 0 } B_a ^n (f) (z,t) &=& \frac{1}{2\pi} \int_\mathbb{Z} \int_\mathbb{R} \sum_{n\geq0} \int e^{(-a+ ib)\sigma(\eta,z)} e^{i\ell \cdot \theta( \eta,\, z)} \hat{f}(\eta.z,\,\ell,\, b)\, d\mathbf{p}^n (\eta)e^{itb}\, db \, d\ell \\
 &=& \frac{1}{2\pi} \int_\mathbb{Z} \int_\mathbb{R} \sum_{n\geq0} P_{-a+ ib, \ell} ^n  \hat{f}( z,\,\ell,\, b) e^{itb}\, db\, d\ell \\ 
  &=& \frac{1}{2\pi} \int_\mathbb{Z} \int_\mathbb{R} (1-P_{-(a- ib,\ell)})^{-1}  \hat{f}(z,\,\ell,\, b) e^{itb}\, db\, d\ell \\
   &=& \frac{1}{2\pi} \int_\mathbb{R} \left( \frac{1}{\chi(a- ib)}N_0+U_0(a- ib ) \right)  \hat{f}(z,\,0,\, b) e^{itb} \,db  \\
   &+& \frac{1}{2\pi} \int_{\mathbb{Z} \setminus \lbrace 0 \rbrace} \int_\mathbb{R} U_\ell(a- ib )   \hat{f}(z,\,\ell,\, b)  e^{itb} \,db\, d\ell.
\end{eqnarray*}
By the identity $\frac{1}{a-ib} = \int_0 ^\infty e^{-(a-ib)u} du$ for $a>0$, and as  $\hat{f_z}(0,\cdot ) \in L^1 (\mathbb{R})$ for every $z \in D$ by \eqref{Eq 4.60}, we obtain
\begin{eqnarray*}
\frac{1}{2\pi} \int_\mathbb{R} \frac{N_0}{\chi(a-ib)}  \hat{f}(z,\,0,\, b) e^{it b} db &=& \frac{1}{2\pi} \frac{1}{\chi} \int_{D} \int_\mathbb{R} \frac{\hat{f}(z,\,0,\, b)}{a-ib} e^{itb} \, db\, d\nu(z)\\
&=& \frac{1}{\chi} \int_{D} \int_0 ^\infty \int_\mathbb{T} f(z,\,y,\, u+t) e^{-au}\, du\, d\nu(z)\, d \lambda_\mathbb{T}(y).\\
\end{eqnarray*}
Taking $a\rightarrow 0^+$, as $f$ is integrable with respect to $\nu\times du \times \lambda_\mathbb{T}$, by monotone convergence this converges to 
$$\frac{1}{\chi}  \int_{D} \int_\mathbb{T}  \int_{-t} ^\infty  f(z,x,y)\,dy\,d\lambda_{\mathbb{T}} (x)\,  d\nu(z).$$ 
This concludes the proof of \eqref{Eq Lemma 4.6 Li}.

Now, by Theorem \ref{Theorem U} and \eqref{Eq 4.60} we have a  bound on the norm of $U_\ell$. So, applying dominated convergence
$$\lim_{a\rightarrow 0^+} \frac{1}{2\pi} \int_\mathbb{Z} \int_\mathbb{R} e^{-it b}  U_\ell(a-ib)\hat{f}( z,\ell, b)\,db \, d\ell =  \frac{1}{2\pi} \int_\mathbb{Z} \int_\mathbb{R} e^{-it b}  U_\ell(-ib)\hat{f}( z,\ell, b)\,db \, d\ell.$$
For every $\ell\in \mathbb{Z}$ we put
$$T_z(\ell, v) = U_\ell(-iv)\hat{f}(z, \ell, v).$$
Thus, $T_z(\ell,\cdot )$ is a tempered distribution with analytic continuation to $|\Im v |<\epsilon$ by  Theorem \ref{Theorem U} and Theorem \ref{Theorem spectral gap} and, so that for all $|b|<\epsilon$ we have $T_z(\ell,\, \cdot + ib)\in L^1$. What is more, by \eqref{Eq 4.60} and Theorem \ref{Theorem U}, for all $|b|<\epsilon$, 
$$\sup_{\eta <|b|} ||T(\ell,\,\cdot + i\eta)||_{L^1}<\infty.$$
So, by \cite[Lemma 4.28]{li2018fourier} we have, for every $0<\delta<\epsilon$
$$\left| \frac{1}{2\pi} \int e^{-it b}  U_\ell(-ib)\hat{f}(z,\ell,b)\,db \right| = \left| \check{T_{z} (\ell,\cdot )}(t) \right| \leq e^{-\delta |t|} \max \left| T_z(\ell, \pm i \delta+b) \right|_{L^1 (b)}.$$
So, by Theorem \ref{Theorem U} and \eqref{Eq 4.60}, for all integer $\ell\neq 0$
$$\max \left|  U_\ell(\ell, -ib\pm\delta)\hat{f}(z, \ell, b\pm i\delta) \right|_{L^1 (b)} \leq  \int \left||U_\ell(\ell, \pm \delta-ib)\right||\cdot  \frac{e^{\delta | \supp f_z|}}{|b|^{4}\cdot |\ell|^4+1}  \left( \left|\frac{\partial}{\partial x^{4} \partial y^{4}} f_z \right|_{L^1} + |f_z|_{L^1} \right) \,db$$
$$ \leq C\cdot  e^{\delta | \supp f_z|} \left( \left|\frac{\partial}{\partial x^{4} \partial y^{4}} f_z \right|_{L^1} + |f_z|_{L^1} \right) \int \frac{1}{|b|^{4}\cdot |\ell|^4+1}\cdot (1+|b|+|\ell|)^{1+\gamma}\, db $$
$$= C\cdot  C_\ell\cdot  O\left(e^{\delta | \supp f_z|} \left( \left|\frac{\partial}{\partial x^{4} \partial y^{4}} f_z \right|_{L^1} + |f_z|_{L^1} \right)  \right), $$
 Noting that a similar bound holds when $\ell= 0$, and that the $C_\ell$s are summable with respect to $\ell$, by \eqref{Eq Lemma 4.6 Li} and the previous discussion, the Proposition follows.

\end{proof}

\subsection{Residue process} \label{Section residue process} In this Section we discuss a certain residue process, that is important in the proof scheme of Theorem \ref{Theorem equi}. This is the only substantially different part of the proof, except for Proposition \ref{Proposition renewal}, compared with that of \cite[Theorem 4.1]{algom2023polynomial}. So, we content ourselves with proving Proposition \ref{Proposition 4.7} below; putting it into the proof scheme of \cite[Theorem 4.1]{algom2023polynomial} would give the proof of Theorem \ref{Theorem equi}, with minor variations.

Fix a non-negative bounded Borel function $f$ on $[D,D']\times \mathbb{T}\times \mathbb{R}$. For all $k\in \mathbb{R}$ and $z\in D$ Let the residue operator be defined as
$$Ef(z, k) := \sum_{n\geq 0} \int \int f(c(i, \eta.z),\,\theta(i, \eta.z)+\theta(\eta,z),\, c(\eta,z)-k) d\mathbf{p}^n (\eta)d\mathbf{p}(i).$$
Recall that for every $z\in D$ we let $f_z :\mathbb{T}\times\mathbb{R} \rightarrow \mathbb{T}\times\mathbb{R}$ be the function 
$$f_z (x,y):=f(z,x,y).$$ 
Here is the main result of this subsection:
\begin{Proposition} \label{Proposition 4.7}  Fix a compact subset  $C\subseteq \mathbb{T} \times  \mathbb{R}$. Let $f$ be a function on $[D,D']\times \mathbb{T} \times \mathbb{R}$ that is non-negative, bounded, compactly supported, and continuous. Suppose $f_w \in C_C ^{8} (\mathbb{T}\times  \mathbb{R})$ for all $w\in [D,D']$, and that $\hat{f}_w (\ell, \theta) \in C^1([D,D'])$ for all $(\ell,  \theta) \in \mathbb{Z}\times \mathbb{R}$. Assume furthermore that
$$\sup_{w}  \left( \left|\frac{\partial}{\partial x^{4} \partial y^{4}} f_w \right|_{L^1} + \left|f_w\right|_{L^1} \right) <\infty.$$
 Then for all $k>0$ and every $z\in D$,
$$Ef(z, k) = \frac{1}{\chi} \int_{-k} ^\infty \int_\mathbb{T} \int_{\mathcal{A}} \int_{D} f( c(i,z),\, s+ \theta(i, z),\, y)\, d\nu (z)\, d\mathbf{p}(i)\, d\lambda_\mathbb{T}(s)\, dy$$
$$+e^{-\epsilon\cdot k}  O\left(e^{\epsilon | \supp f|} \sup_{z,z'\in D} \left( \left|\frac{\partial}{\partial x^{4} \partial y^{4}} f_{z'} \right|_{L^1} + \left|f_z\right|_{L^1} \right)  \right).$$
\end{Proposition}
\begin{proof}
Let  $z \in D,x\in \mathbb{T}$ and let $y\in \mathbb{R}$. We define a function $Qf$  on $D\times \mathbb{T} \times \mathbb{R}$  via
$$Qf(z, x, y) = \int f( c(i,z), x+ \theta(i, z),\, y)d\mathbf{p}(i).$$
Note that $Qf$ is a non-negative, bounded, and Borel function.

Then, with the notations of Section \ref{Section renewal},
$$Ef(z, k) = \sum_{n\geq 0 } \int Qf(\eta.z,\, \theta(\eta,z),\, \sigma(\eta, z)-k) d\mathbf{p}^n (\eta) = R(Qf)(z,k).$$
As $\Phi$ is a $C^\omega (\mathbb{C})$, the function
$$z\mapsto \left(\theta(i, \eta.z),\,c(i,z)) \right)\in C^8(D,\, \mathbb{T}\times [D',D])$$
for every $i\in \mathcal{A}$. Thus, $z\mapsto \hat{f}_{c(i,z)} (\ell, \theta) \in C^1(D)$ for every $i\in \mathcal{A}$ and $(\ell, \theta) \in \mathbb{Z}\times \mathbb{R}$. The result is now a direct application of Proposition \ref{Proposition renewal} for the function  $Qf(z,w,k)$.
\end{proof}

\section{Proof of Theorem \ref{Main Theorem analytic}} \label{Section Fourier decay}
The goal of this Section is prove Theorem \ref{Main Theorem analytic}. Out of all of our work in the previous Sections, we will only make essential use of the equidistribution Theorem \ref{Theorem equi}. Recall that it follows from our previous work that  we may assume any IFS $\Phi$ as in Theorem \ref{Main Theorem analytic}, and any self-conformal measure $\nu_\mathbf{p}$ w.r.t. $\Phi$, satisfy the conditions of Theorem \ref{Theorem equi}, perhaps up to inducing. As usual, we consider the Bernoulli measure $\mathbb{P}=\mathbf{p}^\mathbb{N}$ on the product space $\mathcal{A}^\mathbb{N}$.

We want to show that for some $\alpha>0$
$$\mathcal{F}_q (\nu) = O\left( \frac{1}{ \left| q \right| ^\alpha} \right), \text{ taking } |q|\rightarrow \infty.$$
The proof is based upon \cite[Section 5]{algom2023polynomial}, with some modifications due to our higher dimensional setting.

Fix a large $|q|$ and pick $k=k(q)\approx \log |q|$; the exact choice will be indicated later. Fix the constant $\epsilon>0$ from Theorem \ref{Theorem equi} (the size of the strip where we have spectral gap in Theorem \ref{Theorem spectral gap}).   Consider the (actual) stopping time $\beta_k :\mathcal{A}^\mathbb{N} \rightarrow \mathbb{N}$ defined as
$$\beta_k (\omega):= \min \lbrace m:\, \left| f_{\omega|_m} ' (x_0) \right| < e^{-k-\frac{\epsilon k}{18}} \rbrace, \, \text{ for some prefixed } x_0 \in D.$$
The following linearization Theorem is essentially proved in \cite{algom2020decay}.
 Recall the random variables $S_{\tau_k}$, $\tau_k$ that were defined in Section \ref{Section equi}. For all $t\in \mathbb{C}$ denote by $M_t:\mathbb{C}\rightarrow \mathbb{C}$  the complex multiplication map $M_t (x)=t\cdot x$. 
\begin{theorem}   (Linearization) \label{Theorem linearization}
For all $\beta \in (0,1)$ we have
$$\left| \mathcal{F}_q (\nu) \right|^2 \leq \int \left| \mathcal{F}_q (M_{f'_{\tau_k (\omega)}} \circ f_{\omega|_{\tau_k{(\omega)}+1} ^{\beta_k(\omega)}}  \nu) \right|^2 d\mathbb{P}(\omega)+O\left(|q| e^{-(k+\frac{k\epsilon}{18})-\beta\cdot \frac{k\epsilon}{18}} \right).$$
Additionally, we can find a uniform constant $C'>1$ so that for all $\omega \in \mathcal{A}^\mathbb{N}$
\begin{equation} \label{Eq. uniform push}
\left| f_{\omega|_{\tau_k{(\omega)}+1} ^{\beta_k(\omega)}} ' (x) \right| = \Theta_{C'} \left(e^{-\frac{\epsilon k}{18}} \right).
\end{equation}
\end{theorem}
\begin{proof}
This relies on the following Lemma from \cite{algom2020decay}.  It really only requires the fact that all the maps in $\Phi$ are $C^{1+\gamma}$ for some (here, any) $\gamma>1$, and that $D$ is convex.
\begin{Lemma} \cite[Lemma 2.3]{algom2020decay} For every $\beta\in(0,\gamma)$ there exists $\epsilon\in(0,1)$ such that for all $n\geq 1$, $g\in\Phi^{n}$ and $x,y\in D$ satisfying $|x-y|<\epsilon$, $$\big|g(x)-g(y)-g'(y) (x-y)\big|\leq|g'(y)|\cdot|x-y|^{1+\beta}.$$\end{Lemma}
What the Lemma  means is that for every $y\in B_\epsilon (x)$ the function $g$ may be approximated exponentially fast on $B_\epsilon(x)$ by an affine map with linear part $g'(y)$. Its proof is similar to that of \cite[Lemma 2.3]{algom2020decay}. To be a bit more specific, using Cauchy's integral formula, it is possible to show that
$$\big|g(x)-g(y)-g'(y) (x-y)\big|\leq  C\cdot |x-y|^{1+\beta}$$
for any $x,y\in D$ such that $|x-y|\ll 1$, $g\in \Phi$, and some uniform constant $C=C(\Phi)>0$. This proves  \cite[equation (5)]{algom2020decay}, which is the only part in the proof of \cite[Lemma 2.3]{algom2020decay} that uses a tool that is unavailable in our setting (the mean value Theorem). So, from here, one can simply follow the steps of that proof. Alternatively, one can appeal more directly to Cauchy's formula in order to bound the second derivative with the sup norm over the first derivative.

With the Lemma at our disposal, Theorem \ref{Theorem linearization} follows from a combination of $\mathbb{C}$-type variants of \cite[Lemma 4.3, Lemma 4.4, and Claim 4.5]{algom2020decay}, using the bounded distortion principle \cite[Theorem 2.1]{algom2020decay}  - recall \eqref{Eq bdd distortion}. We omit the details.
\end{proof}

For all $k>0$ consider the measurable partition $\mathcal{P}_k$ of $\mathcal{A}^\mathbb{N}$ defined by
$$\omega \sim_{\mathcal{P}_k} \eta \iff \sigma^{\tau_{k} (\omega)} \omega = \sigma^{\tau_{k} (\eta)} \eta.$$
Let $\mathbb{E}_{\mathcal{P}_k (\xi)} (\cdot )$ denote  expectation w.r.t. the  conditionals of $\mathbb{P}$ on  cells associated with a $\mathbb{P}$-typical $\xi$.  By Theorem \ref{Theorem linearization}, for all $0<\beta<1$,
\begin{equation} \label{Eq gk}
| \mathcal{F}_q (\nu) |^2 \leq \int \mathbb{E}_{\mathcal{P}_k (\xi)} \left( \left| \mathcal{F}_q (M_{f'_{\tau_k (\omega)}} \circ f_{\xi|_{\tau_k{(\xi)}+1} ^{\beta_k(\xi)}}  \nu) \right|^2 \right) d\mathbb{P}(\xi)+O\left(|q| e^{-(k+\frac{k\epsilon}{18})-\beta\cdot \frac{k\epsilon}{18}} \right).
\end{equation}
Additionally, for all $\xi \in \mathcal{A}^\mathbb{N}$,
$$ \left| f_{\xi|_{\tau_k{(\xi)}+1} ^{\beta_k(\xi)}} ' (x) \right| = \Theta_{C'} \left(e^{-\frac{\epsilon k}{18}} \right).$$

Next, fix some $\xi \in \mathcal{A}^\mathbb{N}$ and $k>0$, and consider the smooth function on $\mathbb{T}\times \mathbb{R}$ given by
$$g_{k,\xi} (x,y)=\left| \mathcal{F}_q \left( M_{e^{(-y-k)+i\cdot x}} \circ f_{\xi|_{\tau_k{(\xi)}+1} ^{\beta_k(\xi)}}  \nu \right) \right|^2.$$
Recalling \eqref{Def local norm} and assuming $|q|\cdot e^{-k}\rightarrow \infty$ as $q\rightarrow \infty$,
$$||g_{k,\xi}||_{C^8} \leq O\left( \left( |q|\cdot e^{-k} \right)^8 \right).$$
This bound on the $C^8$ norm holds uniformly over $\xi$, which is critically important for us.  By Theorem \ref{Theorem equi},  for  $\mathbb{P}$-typical $\xi$
$$\mathbb{E}_{\mathcal{P}_k (\xi)} \left( \left| \mathcal{F}_q (M_{f'_{\tau_k (\omega)}}  \circ f_{\xi|_{\tau_k{(\xi)}+1} ^{\beta_k(\xi)}}  \nu) \right|^2 \right) =  \mathbb{E}_{\mathcal{P}_k (\xi)} \left( g_{k, \xi} \left(A_{\tau_k(\omega)},\, S_{\tau_k(\omega)}-k \right) \right) $$
$$= \frac{1}{\chi} \int_{\mathcal{A}^\mathbb{N}} \int_\mathcal{A} \int_\mathbb{T} \int_{-\tilde{c}(i,\omega)} ^0 g_{k,\xi}(s+\tilde{\theta}(i,\omega),x+ \tilde{c}(i,\omega)) \,dx d \lambda_\mathbb{T}(s) d\mathbf{p}(i) d\mathbb{P}(\omega)$$
$$ + e^{-\frac{\epsilon k}{2}} O\left( \left( |q|\cdot e^{-k} \right)^8 \right)$$
$$ =  \frac{1}{\chi} \int_{\mathcal{A}^\mathbb{N}} \int_\mathcal{A} \int_\mathbb{T} \int_{-\tilde{c}(i,\omega)} ^0 g_{k,\xi}(s+\tilde{\theta}(i,\omega),x+\tilde{c}(i,\omega))\,dx d \lambda_\mathbb{T}(s) d\mathbf{p}(i) d\mathbb{P}(\omega)$$
$$+ e^{-\frac{\epsilon k}{2}}\left( |q|\cdot e^{-k} \right)^8 O\left( 1\right).$$

Note that for $f\geq 0$ we have 
$$\frac{1}{\chi} \int_{\mathcal{A}^\mathbb{N}} \int_\mathcal{A} \int_\mathbb{T} \int_{-\tilde{c}(i,\omega)} ^0 f(s+\tilde{\theta}(i,\omega),x+\tilde{c}(i,\omega))\,dx d \lambda_\mathbb{T}(s) d\mathbf{p}(i) d\mathbb{P}(\omega) \leq \frac{1}{\chi} \int_{0} ^{D'} \int_\mathbb{T}  f(s,t) d\lambda_\mathbb{T}(s) \,dt, $$
So, using the previous equality with \eqref{Eq gk}, and by the definition of $g_{k,\xi}$ we have:
\begin{eqnarray*}
\left| \mathcal{F}_q (\nu) \right|^2 &\leq& \int \mathbb{E}_{\mathcal{P}_k (\xi)} \left( \left| \mathcal{F}_q (M_{f'_{\tau_k (\omega)}}  \circ f_{\xi|_{\tau_k{(\xi)}+1} ^{\beta_k(\xi)}}  \nu) \right|^2 \right) d\mathbb{P}(\xi)+O\left(|q| e^{-(k+\frac{k\epsilon}{18})-\beta\cdot \frac{k\epsilon}{18}} \right) \\
&= & \frac{1}{\chi} \int_{\mathcal{A}^\mathbb{N}} \int_\mathcal{A} \int_\mathbb{T} \int_{-\tilde{c}(i,\omega)} ^0 g_{k,\xi}(s+\tilde{\theta}(i,\omega),x+\tilde{c}(i,\omega))\,dx d \lambda_\mathbb{T}(s) d\mathbf{p}(i) d\mathbb{P}(\omega) d\mathbb{P}(\xi)\\
&+&O\left(|q| e^{-(k+\frac{k\epsilon}{18})-\beta\cdot \frac{k\epsilon}{18}} \right)+e^{-\frac{\epsilon k}{2}}\left( |q|\cdot e^{-k} \right)^8 O\left( 1\right)\\
&\leq  & \frac{1}{\chi}  \int \int_{0} ^{D'} \int_\mathbb{T} \left| \mathcal{F}_q \left( M_{e^{(-y-k)+ix}} \circ f_{\xi|_{\tau_k{(\xi)}+1} ^{\beta_k(\xi)}}  \nu \right) \right|^2 d\lambda_\mathbb{T}(x)\, dy d\mathbb{P}(\xi)+O\left(|q| e^{-(k+\frac{k\epsilon}{18})-\beta\cdot \frac{k\epsilon}{18}} \right)\\
&&+e^{-\frac{\epsilon k}{2}}\left( |q|\cdot e^{-k} \right)^8 O\left( 1\right).
\end{eqnarray*}

Finally, we require the following  Lemma of Hochman \cite{Hochman2020Host} to treat this oscillatory integral.
\begin{Lemma} \cite[Lemma 2.6]{algom2020decay} (Oscillatory integral) For all $\xi \in  \mathcal{A}^\mathbb{N}$, every $k>0$, and all $r>0$
$$\int_0 ^{D'} \int_\mathbb{T} \left| \mathcal{F}_q \left( M_{e^{(-y-k)+ix}} \circ f_{\xi|_{\tau_k{(\xi)}+1} ^{\beta_k(\xi)}}  \nu \right) \right|^2\, d\lambda_\mathbb{T}(x)\,dy  =O\left( \frac{1}{r|q|e^{-(k+\frac{\epsilon k}{18})}}+\sup_{y} \nu(B_r (y)) \right)$$
\end{Lemma}
We remark that this Lemma is proved in \cite{Hochman2020Host} for the $1$-dimensional case; the proof in dimension $2$ is similar. Also, we use \eqref{Eq. uniform push} to bound uniformly  the first term on the RHS.
$$ $$

\noindent{\textbf{Conclusion of proof}} We have just shown that $\left| \mathcal{F}_q (\nu) \right|^2$ can be bounded by the sum of several terms that we now discuss. These bounds depend  implicitly  on $\mathbf{p}$ and $\Phi$. Also, we ignore global multiplicative constants,  omitting  big-$O$ notations.

Linearization: For all $\beta\in (0,1)$ (that is up to us to choose),
$$|q| e^{-(k+\frac{k\epsilon}{18})-\beta\cdot \frac{k\epsilon}{18}};$$
Equidistribution: 
$$e^{-\frac{\epsilon k}{2}}\left( |q|\cdot e^{-k} \right)^8 ;$$
Oscillatory integral: For all $r>0$
$$\frac{1}{r|q|e^{-(k+\frac{\epsilon k}{18})}}+\sup_{y} \nu(B_r (y)).$$

\noindent{\textbf{Choice of parameters}} Given $|q|$ we pick $k=k(q)$ so that 
$$|q|=e^{k+\frac{k\epsilon}{17}}.$$
Also, let $r=e^{-\frac{k \epsilon}{1000}}$ and pick $\beta = \frac{1}{2}$.  Thus:

Linearization:
$$|q| e^{-(k+\frac{k\epsilon}{18})-\beta\cdot \frac{k\epsilon}{18}} = e^{\frac{k \epsilon}{17} -\frac{k\epsilon}{18}- \frac{k\epsilon}{36}},\, \text{ this decays exponentially fast in }k.$$
Equidistribution: 
$$e^{-\frac{\epsilon k}{2}}\left( |q|\cdot e^{-k} \right)^8 =  e^{-\frac{\epsilon k}{2} +\frac{8k\epsilon}{17}},\, \text{ this decay exponentially fast in }k.$$
Oscillatory integral: For some $d=d(\nu)>0$ we have
$$\frac{1}{r|q|e^{-(k+\frac{\epsilon k}{18})}} + \sup_{y} \nu(B_r (y)) \leq  \frac{1}{e^{-\frac{k\epsilon}{1000}+ \frac{k \epsilon}{17} -\frac{k\epsilon}{18}}}+e^{-\frac{d \epsilon k}{1000}},\, \text{ this decay exponentially fast in }k.$$
Here we  used  \cite[Proposition 2.2]{Feng2009Lau}, where it is proved that for some $C>0$, for all $r>0$ small enough, $\sup_{y} \nu(B_r (y))\leq Cr^d$. 

Finally, summing the error terms, we can see that for some $\alpha>0$,
$$ \mathcal{F}_q (\nu)  = O\left(e^{-k\alpha}\right).$$
Notice that as $|q|\rightarrow \infty$,  $k\geq C_0 \cdot \log |q|$ for a uniform constant $C_0>0$. Our claim follows.   \hfill{$\Box$}

\section{Acknowledgements}
We thank Amlan Banaji, Simon Baker, Tuomas Sahlsten,  and Osama Khalil, for  their comments on this project. We also thank the referee for useful remarks and for pointing out many typos.  This research was supported by Grant No. 2022034 from the United States - Israel Binational Science Foundation 
(BSF), Jerusalem, Israel.

\bibliography{bib}{}
\bibliographystyle{plain}

\end{document}